\documentclass{article}
\usepackage{authblk}

\usepackage{amsmath}
\usepackage{amssymb}
\usepackage{amsthm}
\allowdisplaybreaks

\usepackage{hyperref}

\usepackage{mathtools}

\usepackage{mismath}

\usepackage{bm} % Bold math font.
\usepackage{bbm}
\usepackage{graphicx}
\graphicspath{{figs/}} % Setting the graphicspath

\usepackage{tikz}
\usetikzlibrary{math}
\usetikzlibrary{calc}

\usepackage{subcaption}

\usepackage{multirow}

\usepackage{siunitx}
\usepackage{makecell}
\setcellgapes{3pt}

\usepackage{algorithm}

\usepackage{algpseudocode}

\usepackage[capitalize]{cleveref}
\usepackage{easyReview}
\setreviewsoff

\newcommand{\Real}{\mathbb{R}}

\newcommand{\TD}{\Gamma_{\rm{D}}}
\newcommand{\TN}{\Gamma_{\rm{N}}}

\newcommand{\TC}{\Gamma_{\rm{C}}}

\DeclareMathOperator{\Div}{div}

\DeclareMathOperator{\Span}{span}

\DeclarePairedDelimiter{\CurlyBrackets}{\{}{\}}

\newtheorem{theorem}{Theorem}[section]
\newtheorem{lemma}[theorem]{Lemma}
\newtheorem{corollary}[theorem]{Corollary}
\newtheorem{proposition}[theorem]{Proposition}
\theoremstyle{definition}
\newtheorem{definition}{Definition}[section]

\theoremstyle{remark}
\newtheorem*{remark}{Remark}

\definecolor{MyColor1}{HTML}{CECE5A}
\definecolor{MyColor2}{HTML}{C51605}
\definecolor{MyColor3}{HTML}{B7B7B7}
\definecolor{MyColor4}{HTML}{FFE17B}

\usepackage{array}
\newcolumntype{L}[1]{>{\raggedright\let\newline\\\arraybackslash\hspace{0pt}}m{#1}}
\newcolumntype{C}[1]{>{\centering\let\newline\\\arraybackslash\hspace{0pt}}m{#1}}
\newcolumntype{R}[1]{>{\raggedleft\let\newline\\\arraybackslash\hspace{0pt}}m{#1}}

\usepackage{booktabs}

\providecommand{\keywords}[1]
{
	\small	
	\textbf{\textit{Keywords---}} #1
}

\usepackage[twoside,
paperwidth=210mm,
paperheight=297mm,
textheight=622pt,
textwidth=468pt,
centering,
headheight=50pt,
headsep=12pt,
footskip=18pt,
footnotesep=24pt plus 2pt minus 12pt,
columnsep=2pc]{geometry}

\title{Multiscale modeling for contact problem with high-contrast heterogeneous coefficients with primary-dual formulation}
\author[1]{Zishang Li}
\author[1]{Changqing Ye}
\author[1]{Eric T.~Chung\thanks{Corresponding author:\href{mailto:zsli@math.cuhk.edu.hk}{tschung@math.cuhk.edu.hk}}}

\affil[1]{{Department of Mathematics, The Chinese University of Hong Kong, Shatin, Hong Kong SAR}}

\begin{document}
\maketitle

\begin{abstract}
	In this paper, we propose a novel iterative multiscale framework for solving high-contrast contact problems of Signorini type.
The method integrates the constrained energy minimizing generalized multiscale finite element method (CEM-GMsFEM) with a primal-dual active set strategy derived from semismooth Newton methods.
First, local spectral problems are employed to construct an auxiliary multiscale space, from which energy minimizing multiscale basis functions are derived on oversampled domains, yielding a contrast-robust reduced-order approximation of the underlying partial differential equation.
The multiscale bases are updated iteratively, but only at contact boundary, during the active set evolution process.
Rigorous analysis is provided to establish error estimates and finite step convergence of the iterative scheme. Numerical experiments on heterogeneous media with high-contrast coefficients demonstrate that the proposed approach is both robust and efficient in capturing fine-scale features near contact boundaries.

\end{abstract}


\section{Introduction}
Composite materials have received extensive attention in modern engineering applications due to their superior strength-to-weight ratios and enhanced performance characteristics.
However, the analysis of composite structures is inherently challenging because physical problems related to these materials exhibit multiscale features and high-contrast material properties.
In many cases, the direct use of conventional finite element methods (FEM) demands extremely fine discretizations, leading to large-scale systems and prohibitive computational costs.

To overcome these difficulties, various multiscale methods have been introduced. In particular, methods such as the multiscale finite element method (MsFEM) \cite{Chen2003,Efendiev2009,Hou1997,Hou1999}, heterogeneous multiscale method (HMM) \cite{Abdulle2012,Ming2005,Weinan2007} and generalized multiscale finite element methods (GMsFEM) \cite{Efendiev2013,Chung2014,Chung2016} have demonstrated the ability to efficiently capture fine-scale heterogeneities by embedding small-scale information into a reduced-order subspace.
The constrained energy minimizing GMsFEM (CEM-GMsFEM) further enhances this strategy by systematically constructing multiscale basis functions via local energy minimization.
This approach guarantees that the constructed basis exhibits exponential decay away from local support regions and provides robust accuracy even in high-contrast settings.

In practical engineering applications, additional complexities arise when contact phenomena occur.
In general, contact problems are formulated as variational inequalities that originate from geometric constraints along the contact interface and are inherently nonsmooth in nature \cite{Duvaut1976}.
For frictionless contact problems---those that can be recast as convex optimization problems---the mathematical framework is well established \cite{Lions1967}; however, for frictional contact problems the theoretical understanding remains less complete \cite{Eck2005,Ballard2024}.

From a computational perspective, it is crucial to carefully address both the discretization of the governing equations and the development of efficient numerical algorithms.
Two primary discretization strategies are commonly employed. First, Galerkin methods impose the constraints directly in their primal formulation \cite{Hild2012,Drouet2015}.
Second, saddle-point methods explicitly account for the coupling between the primary variable and its associated Lagrange multiplier \cite{Wohlmuth2000,Hueeber2008,Wohlmuth2011}.
Even in the case of Galerkin formulations, designing robust solvers for the resulting systems of inequalities remains challenging.
Widely used techniques include the penalty method, which incorporates a penalty term into the objective function to weakly enforce the constraints \cite{Wu2022}, and methods that exploit the problem's primal-dual structure to yield a saddle-point formulation \cite{Wriggers2006}.
Moreover, when the underlying model exhibits multiscale features, the numerical treatment of contact problems becomes even more complicated; see \cite{Ye2021a,Ye2023a} for recent homogenization results.
Therefore, special attention must be paid to the design of multiscale computational methods for these nonsmooth boundary value problems.

In this work, we propose a novel iterative multiscale framework that integrates the CEM-GMsFEM with a primal-dual approach for solving high-contrast contact problems.
The proposed framework first constructs an auxiliary space via local spectral problems and subsequently builds multiscale basis functions by solving localized energy minimization problems on oversampled domains.
These basis functions are then employed to construct a contrast-robust approximation of the solution to the underlying PDE, which in turn defines a reduced-order space for the contact problem.
Owing to the nonlinear nature of the problem, the multiscale bases are updated only at the contact boundary by appropriately restricting certain degrees of freedom.
Compared with our previous work \cite{Li2025}, the current approach avoids the use of penalty methods, thereby circumventing the accuracy deterioration often associated with the inappropriate choice of penalty parameters.

The remainder of the paper is organized as follows.
Section~2 formulates the model problem and reviews the underlying variational inequality and its constrained minimization formulation.
In Section~3, we describe the construction of the CEM-GMsFEM basis functions and detail the primary-dual active set method used for solving the resulting variational inequality problem.
Section~4 presents the theoretical analysis, including error estimates and convergence properties, which elucidate the advantages of the proposed method.
Finally, Section~5 demonstrates the performance of our approach through several numerical experiments involving heterogeneous media and contact boundaries.

\section{Preliminaries}\label{sec:pre}
\subsection{Model problem}
We will consider the contact problem of Signorini type for a second-order elliptic partial differential equation. We denote a Lipschitz domain by $\Omega \subset \R^d$ ($d=2\ \text{or}\ 3$), and $\kappa\in L^\infty\left(\Omega ;\ \R^{d \times d} \right) $ a matrix-valued function defined on $ \Omega $ represents a heterogeneous permeability field with high contrast:
\begin{equation}\label{eq:strong}
  \left\{
  \begin{aligned}
     & - \Div\left( {\kappa \nabla u} \right) = f                                            &  & \text { in } \Omega, \\
     & u=0                                                                                   &  & \text { on } \TD,    \\
     & \kappa\nabla u\cdot\bm{n}=p                                                           &  & \text { on } \TN,    \\
     & u \leqslant0, \ \kappa\nabla u\cdot\bm{n} \leqslant0,\ u(\kappa\nabla u\cdot\bm{n})=0 &  & \text { on } \TC,    \\
  \end{aligned}
  \right.
\end{equation}
where $\bm{n}$ is the outward unit normal to $\partial\Omega$, $\TD$, $\TN$ and $\TC$ are three nonempty disjointed parts of $\partial\Omega$. We assume that the source term $f \in L^2(\Omega ) $, and there exist two positive constants $ \kappa' $ and $ \kappa'' $ such that  $0<\kappa' \leqslant \kappa(\bm{x}) \leqslant \kappa''<\infty$ for almost all $ \bm{x}\in\Omega $ and the contrast ratio $ \kappa'/\kappa'' $ could be very large.

To solve this problem, it is appropriate to use a functional framework that involves a subset of the Sobolev space $H^1(\Omega)$ defined as
\[ V \coloneqq\left\{v \in H^1(\Omega): \  v=0\text{ on } \TD \right\}. \]
The contact condition is then explicitly incorporated in the following closed convex set
\[ K\coloneqq\left\{v \in V: \  v\leqslant0 \text{ on } \TC\right\}.\]
By the primal variational principle for the Signorini problem, the exact solution of the above contact problem (\cref{eq:strong}) is characterized by the variational inequality: find $u \in K$ such that
\begin{equation}\label{eq:var inequ}
  a(u, v-u) \geqslant L(v-u), \  \forall v \in K,
\end{equation}
where
\[  a(u, v) =\int_{\Omega}\kappa \nabla u \cdot\nabla v \di \bm{x} \ \text{and}\ L(v) =\int_{\Omega} f v \di \bm{x}+\int_{\Gamma_N} p v \di \sigma. \]
%\[ \begin{aligned}
%a(u, v) & =\int_{\Omega}\kappa \nabla u \cdot\nabla v \di \bm{x}, \\
%L(v) & =\int_{\Omega} f v \di \bm{x}+\int_{\Gamma_N} p v \di \sigma.
%\end{aligned} \]
We introduce the notation for the energy norm $ {\left\| v \right\|_a} \coloneqq \sqrt {a(v,v)}$ on $ \Omega $. For a subdomain $ \omega \subset \Omega $ , we also introduce the norm $ {\left\| v \right\|_{a(\omega)}} \coloneqq \sqrt{\int_{\omega}\kappa \nabla u \cdot\nabla v \di \bm{x}}. $

By Theorem 3.9 (see \cite{Kikuchi1988}), the weak problem (\cref{eq:var inequ}) is well-posed and the unique solution of this variational inequality exists. Moreover, this solution is also the minimizer of the constrained minimization problem: given the functional $F: V \rightarrow \R$ is of the form
\[F(v)=\frac{1}{2} a(v, v)-L(v),\]
find $u \in K$ such that
\begin{equation}\label{eq:con min}
  F(u) = \inf\limits_{v \in K} F(v).
\end{equation}
It can be considered as the following constrained minimization problem:
\begin{equation}\label{eq:min_subject}
  \left\{
  \begin{aligned}
     & \min \frac{1}{2} a(u,u)-L(u),                          \\
     & \text{subject to } u\in V,\ u\leqslant0\text{ on }\TC. \\
  \end{aligned}
  \right.
\end{equation}
While the penalty method approximates a constrained problem by adding a term that penalizes constraint violations to the objective function, the Lagrange multiplier method exactly encodes the constraints into the optimality conditions through the introduction of auxiliary variables (the multipliers). The key superiority of the Lagrange multiplier method lies in its ability to precisely satisfy the constraints at the solution without the numerical ill-conditioning associated with driving a penalty parameter to infinity. For constrained optimization problems, the method of Lagrange multipliers stands as a fundamental strategy for constrained optimization due to its profound conceptual and practical advantages. Its primary superiority lies in transforming a constrained minimization problem into a higher-dimensional unconstrained problem via the Lagrangian function. This transformation not only simplifies the problem but also provides a systematic approach to handle constraints, making it particularly effective for problems with multiple or complex constraints. For inequality-constrained optimization problems, this method can derive necessary conditions for optimality through the Karush-Kuhn-Tucker (KKT) conditions, thereby enhancing its practicality and providing a clear framework for identifying optimal solutions within constrained environments. When the objective function and constraints are convex, the KKT conditions also serve as sufficient conditions for finding the optimal solution. Thus, the method of Lagrange multipliers is a versatile and powerful tool that significantly contributes to the field of optimization by facilitating the resolution of constrained problems with elegance and efficiency. Associated to \cref{eq:min_subject} we define the Lagrangian $\mathcal{L}: V \times H^{-\frac{1}{2}}(\TC)\to \Real $ by
\begin{equation}\label{eq:}
  \mathcal{L}(u, \lambda) = \frac{1}{2} a(u, u)-L(u)+\langle {\lambda,u}\rangle _{\TC} ,
\end{equation}
where and $ \langle {\cdot,\cdot}\rangle _{\TC} $  denotes the duality pairing between $ H^{\frac{1}{2}}(\TC) $ and $ H^{-\frac{1}{2}}(\TC) $. \cref{eq:min_subject} admits a unique solution denoted by $ u^* \in V $. There exists a Lagrange multiplier $ \lambda^* \in H^{-\frac{1}{2}}(\TC) $ which renders $ \mathcal{L} $ stationary at $ (u^*,\lambda^*) $, i.e., $\exists \ (u^*,\lambda^*) $ such that
\begin{equation}\label{eq:stationary}
  \left\{
  \begin{aligned}
     & \int_{\Omega}\kappa \nabla u \cdot\nabla v \di \bm{x}+ \langle {\lambda,v}\rangle _{\TC}=\int_{\Omega} f v \di \bm{x} +\int_{\Gamma_N} p v \di \sigma  \  &  & \text{for all } v\in V,                        \\
     & \langle {\mu-\lambda,u}\rangle_{\TC}\geqslant0,\ \mu\geqslant0                                                                                            &  & \text{for all } \mu \in H^{-\frac{1}{2}}(\TC), \\
  \end{aligned}
  \right.
\end{equation}
Let $\mu=0$ and $\mu=2\lambda$ lead to the KKT conditions for this constrained minimization problem:
\begin{equation}\label{eq:kkt}
  \left\{
  \begin{aligned}
     & \int_{\Omega}\kappa \nabla u \cdot\nabla v \di \bm{x}+ \langle {\lambda,v}\rangle _{\TC}=\int_{\Omega} f v \di \bm{x} +\int_{\Gamma_N} p v \di \sigma  \  &  & \text{for all } v\in V,                        \\
     & \langle {\mu,u}\rangle_{\TC}\leqslant0,\ \mu \geqslant 0,\ \langle {\lambda,u}\rangle _{\TC}=0 \                                                          &  & \text{for all } \mu \in H^{-\frac{1}{2}}(\TC), \\
  \end{aligned}
  \right.
\end{equation}
which can formally be expressed as
\begin{equation}\label{eq:formally}
  \left\{
  \begin{aligned}
     & -\Div\left( {\kappa \nabla u} \right) = f,       &  & \text { in } \Omega, \\
     & u=0,                                             &  & \text { on } \TD,    \\
     & \kappa\nabla u\cdot\bm{n}=p,                     &  & \text { on } \TN,    \\
     & \kappa\nabla u\cdot\bm{n}=-\lambda,              &  & \text { on } \TC.    \\
     & u\leqslant0,\ \lambda \geqslant 0,\ \lambda u=0, &  &                      \\
  \end{aligned}
  \right.
\end{equation}

\subsection{Semismooth Newton method}\label{subsec:semi-new}
We denote
\begin{equation}\label{eq:F(lambda)}
  F(\lambda)=\lambda-\max(0,\lambda+cu|_{\TC}).
\end{equation}
Note that the complementarity system given by the last line in \cref{eq:formally} can equivalently be expressed as $ F(\lambda)=0 $ for each $ c>0 $, where max denotes the pointwise a.e.maximum operation. We now briefly recall the facts on semi-smooth Newton methods and find the solution to $ F(\lambda)=0 $.

Let $ X $ and $ Z $ be Banach spaces and let $ F: D \subset X \to Z $ be a nonlinear mapping with open domain $ D $.
\begin{theorem}
  Let $ F: D\to \Real $. Then for $ x\in D$, the following statements are equivalent:\\
  (a) $ F $ is semismooth at $ x $.\\
  (b) $ F $ is locally Lipschitz continuous at $ x $, $  F'(x;\cdot) $ exists, and for any $ G\in \partial F(x + d) $,
  \[\norm{Gd-F'(x, d)} = \bigO(\norm{d})\text{ as } d\to 0 .\]
  (c) $ F $ is locally Lipschitz continuous at $ x $, $  F'(x;\cdot) $ exists, and for any $ G\in \partial F(x + d) $,
  \[\norm{F(x+d)-F(x)-Gd}=\bigO(\norm{d}) \text{ as } d\to 0 .\]
\end{theorem}
Here $ F'(x; d) $ and $ \partial F(x+d) $ are the directional derivative and the collection of the generalized directional derivative of $ F $ at $ x $ in direction $ d $ respectively. The definition can be found in \cite{Hintermueller2010}. Then a generalized derivative is defined for the semismooth Newton procedure.

\begin{definition}
  The mapping $ F:D \subset X \to Z $ is Newton differentiable on the open set $ U \subset D $, if there exists a family of mappings $ G  : U \to L(X,Z)$ such that
  \[\lim\limits_{d \to 0}\frac{1}{\norm{d}} \norm{F(x+d)-F(x)-G(x+d)d} =0. \]
  for every $ x \in U $. The operator $ G $ is referred to as Newton's derivative of $ F $.
\end{definition}

\begin{theorem}\label{thm:Newton}
  Suppose that $ x^*\in D $ is a solution to $ F(x)=0 $ and that $F$ is Newton-differentiable in an open neighborhood $ U $ containing $ x^* $ and that $ \left\{{\left\lVert {G(x)^{-1}}\right\rVert  : x\in U }\right\}$ is bounded. Then the Newton-iteration $ x_{k+1} = x_k-G(x_k)^{-1}F(x_k) $ converges superlinearly to $ x^* $ provided that $ \left\lVert {x_0-x^*} \right\rVert$ is sufficiently small.
\end{theorem}
Let us consider Newton-differentiability of the max-operator. For this purpose $ X $ denotes a function space of real-valued functions on $ \Omega \subset \Real^n $  and $ \max(0,\lambda) $  is the pointwise max-operation. For $ \delta \in \Real $ we introduce candidates for the generalized derivative of the form
\begin{equation}\label{eq:derivative of max}
  G_\delta(\lambda)(x) = \left\{
  \begin{aligned}
     & 0,      &  & \text{if } \lambda(x) < 0, \\
     & \delta, &  & \text{if } \lambda(x) = 0, \\
     & 1,      &  & \text{if } \lambda(x) > 0. \\
  \end{aligned}
  \right.
\end{equation}
\begin{proposition}\label{prop:max-operator}
  The mapping $\max(0,\cdot): L^q(\Omega) \to L^p(\Omega) $  with $ 1\leqslant p<q<\infty $  is Newton differentiable on $ L^q(\Omega) $ and $ G_\delta(\lambda) $ is a Newton's derivative.
\end{proposition}
For the proofs of \cref{thm:Newton} and \cref{prop:max-operator} we refer to \cite{Hintermueller2002}. The following chain rule will be utilized to derive the Newton's derivative of $F(\lambda)$.
\begin{proposition}\label{prop:chain rule}
  Let $ F_2: Y \to X $ be an affine mapping with $ F_2 y = By + b $, $ B \in L(Y, X) $, $ b \in X $, and assume that $ F_1: D\subset X \to Z $  is Newton-differentiable on the open subset $ U\subset D$ with Newton's derivative $ G $. If $ F^{-1}_2(U) $ is nonempty then $ F = F_1 \circ  F_2 $  is Newton-differentiable on $ F^{-1}_2(U ) $ with Newton's derivative given by $ G(By + b)B \in L(Y, Z) $, for $ y\in F^{-1}_2(U )$.
\end{proposition}

% \begin{equation}\label{eq:formally F}
%   \left\{
%   \begin{aligned}
%      & -\Div\left( {\kappa \nabla u} \right) = f,      &  & \text { in } \Omega, \\
%      & u=0,                                            &  & \text { on } \TD,    \\
%      & \kappa\nabla u\cdot\bm{n}=p,                    &  & \text { on } \TN,    \\
%      & \kappa\nabla u\cdot\bm{n}=-\lambda,             &  & \text { on } \TC.    \\
%      & F(\lambda)=\lambda-\max(0,\lambda+cu|_{\TC})=0, &  &                      \\
%   \end{aligned}
%   \right.
% \end{equation}
We can see that $ F(\lambda) $ is semismooth at $\lambda$. By the \cref{prop:chain rule}, the Newton's derivative $ G_F $  of $ F(\lambda) $ is
\begin{equation}\label{eq:derivative of F}
  G_F(\lambda) = 1-(1+cu'(\lambda))\chi_{\mathcal{A}_{\lambda}},
\end{equation}
where without loss of generality let $\delta=0$ and $ \chi_{\mathcal{A}_{\lambda}} $ be the characteristic function of the set
\[\mathcal{A}_{\lambda} = \left\{x\in\TC: \lambda+cu > 0\right\}.\]
To express the set in the Newton step $ \lambda_{k+1} = \lambda_k-G(\lambda_k)^{-1}F(\lambda_k) $, we denote the active set $ \mathcal{A}_k $ as
\[\mathcal{A}_k = \left\{x\in\TC: \lambda_k+cu_k > 0\right\},\]
and the inactive set $ \mathcal{I}_k $ as $ \mathcal{I}_k=\TC \setminus \mathcal{A}_k $. The Newton-iteration for $ F(\lambda) $ is derived as
\begin{equation*}
  \begin{aligned}
    (\lambda_{k+1} -\lambda_k)G_F(\lambda_k)                                                                                              & = -F(\lambda_k)                                \\
    \lambda_{k+1} -\lambda_k -(\lambda_{k+1} -\lambda_k)\chi_{\mathcal{A}_k}-c\chi_{\mathcal{A}_k}u'(\lambda_k)(\lambda_{k+1} -\lambda_k) & = -\lambda_k+ \max(0,\lambda_k+cu(\lambda_k)). \\
  \end{aligned}
\end{equation*}
Transposing terms and using the first-order Taylor series approximation of $u(\lambda_k)$, we obtain
$$ \lambda_{k+1} \approx (\lambda_{k+1} -\lambda_k)\chi_{\mathcal{A}_k}+\chi_{\mathcal{A}_k}(cu(\lambda_{k+1})-cu(\lambda_k))+\max(0,\lambda_k+cu(\lambda_k)) $$
Then combining like terms for the right-hand side,
\begin{equation}\label{eq:iteration-lambda}
  \lambda_{k+1} = (\lambda_{k+1} +cu(\lambda_{k+1}))\chi_{\mathcal{A}_k},
\end{equation}
Thus, when $ x\in \mathcal{A}_k$, $u(\lambda_{k+1})=0,$ and when $ x\in \mathcal{I}_k=\TC \setminus \mathcal{A}_k $, $ \lambda_{k+1}=0$. For simplify, we take $u_k=u(\lambda_k)$, then $ u_{k+1}, \lambda_{k+1}$ are the solution to
% \begin{equation}\label{eq:formally iteration}
%   \left\{
%   \begin{aligned}
%      & -\Div\left( {\kappa \nabla u} \right) = f,                                   &  & \text { in } \Omega,        \\
%      & u=0,                                                                         &  & \text { on } \TD,           \\
%      & \kappa\nabla u\cdot\bm{n}=p,                                                 &  & \text { on } \TN,           \\
%      & \lambda_{k+1}=-\kappa\nabla u\cdot\bm{n}=0,                                  &  & \text { on } \mathcal{I}_k, \\
%      & \lambda_{k+1}=-\kappa\nabla u\cdot\bm{n}=\textcolor{red}{\lambda_{k+1}} +cu, &  & \text { on } \mathcal{A}_k, \\
%     %  & \lambda_{k+1} = 0,                                           &  & \text { on } \mathcal{I}_k,  \\
%     %  & \lambda_{k+1} = \lambda_{k+1} +cu(\lambda_{k+1}),            &  & \text { on }  \mathcal{A}_k. \\
%   \end{aligned}
%   \right.
% \end{equation}
\begin{equation}\label{eq:formally iteration-1}
  \left\{
  \begin{aligned}
     & -\Div\left( {\kappa \nabla u} \right) = f, &  & \text { in } \Omega,        \\
     & u=0,                                       &  & \text { on } \TD,           \\
     & \kappa\nabla u\cdot\bm{n}=p,               &  & \text { on } \TN,           \\
     & -\kappa\nabla u\cdot\bm{n}=0,              &  & \text { on } \mathcal{I}_k, \\
     & \lambda = 0,                               &  & \text { on } \mathcal{I}_k, \\
     & u=0,                                       &  & \text { on } \mathcal{A}_k, \\
     & \lambda = -\kappa\nabla u\cdot\bm{n},      &  & \text { on } \mathcal{A}_k. \\
  \end{aligned}
  \right.
\end{equation}
We rewrite this Dirichlet-Neumann boundary value problem \cref{eq:formally iteration-1} in a variational form for designing computational methods: find a solution $u_{k+1} \in V_k \coloneqq\left\{v \in H^1(\Omega): \  v=0\text{ on } \TD \cup \mathcal{A}_{k}\right\}$ and $\lambda_{k+1} \in H^{-\frac{1}{2}}(\TC)$
such that for all $v \in  V_k$,
\begin{equation}\label{eq:variational iteration-1}
  \int_{\Omega}\kappa \nabla u_{k+1} \cdot\nabla v \di \bm{x} =\int_{\Omega} f v \di \bm{x}+\int_{\Gamma_N} p v \di \sigma \text{ and } \lambda_{k+1}=-\kappa\nabla u_{k+1}\cdot\bm{n} \text{ on } \TC.
\end{equation}
The semi-smooth Newton method can now be expressed as the following primal-dual active set strategy in \cref{algo}:
\begin{algorithm}
  \caption{Primal-dual active set algorithm} \label{algo}%算法的名字
  \hspace*{0.02in} {\bf Input:} %算法的输入， \hspace*{0.02in}用来控制位置，同时利用 \\ 进行换行
  coefficient $\kappa$, source term $ f $, Neumann boundary term $ p $, constant $ c>0 $%算法的输出 
  \begin{algorithmic}[1]
    \State Initialize $u_0$, $\lambda_0$. Set $k=0$.
    \State Set $\mathcal{A}_k=\left\{x\in\TC: \lambda_k+cu_k> 0\right\}$, $\mathcal{I}_k=\TC \setminus \mathcal{A}_k$.
    \State Solve \cref{eq:variational iteration-1} to obtain $u_{k+1}$ and $\lambda_{k+1}$.
    \State Set $\mathcal{A}_{k+1}=\left\{x\in\TC: \lambda_{k+1}+cu_{k+1} > 0\right\}$, $\mathcal{I}_{k+1}=\TC \setminus \mathcal{A}_{k+1}$.
    \State Stop if $\mathcal{A}_{k+1}=\mathcal{A}_k$ or set $k\coloneqq k+1$ and go to step 2.
    % \While{$ k<1 $ or $ R(u_{k+1})>\mathup{tol} $}
    % \State Solve $ G(u_k)d_k=R(u_k) $ to obtain $ d_k $
    % \State $ u_{k+1} = u_k-d_k $, $ k = k+1 $
    % \EndWhile
    % \State \Return $u_{k+1}$
  \end{algorithmic}
  \hspace*{0.02in} {\bf Output:} %算法的结果输出
  iter number $(k+1)$, $u_{k+1}$, $\lambda_{k+1}$.
\end{algorithm}\\

\section{Numerical method}\label{sec:method}

% \subsection{One step CEM-GMsFEM solver}\label{subsec:cemgmsfem}
Consider a conforming partition $\mathcal{T}^H$ of a domain $\Omega$ into $ N $ finite elements $K_i$, where $H$ denotes the coarse-mesh size. This is to be distinguished from another fine mesh $\mathcal{T}^h$ with the mesh size $ h $ and will be utilized to compute multiscale basis functions. The fine element is represented as $ \tau $. For each coarse element $K_i \in \mathcal{T}^H$ with $1 \leqslant i \leqslant N$, we define an oversampling domain $K^m_i(m \geqslant 1)$ as the domain obtained by augmenting the coarse element $K_i$ with an additional $m$ layers of neighboring coarse elements. A representation of the fine grid, coarse grid, and oversampling domain is provided in \cref{fig:grid}.
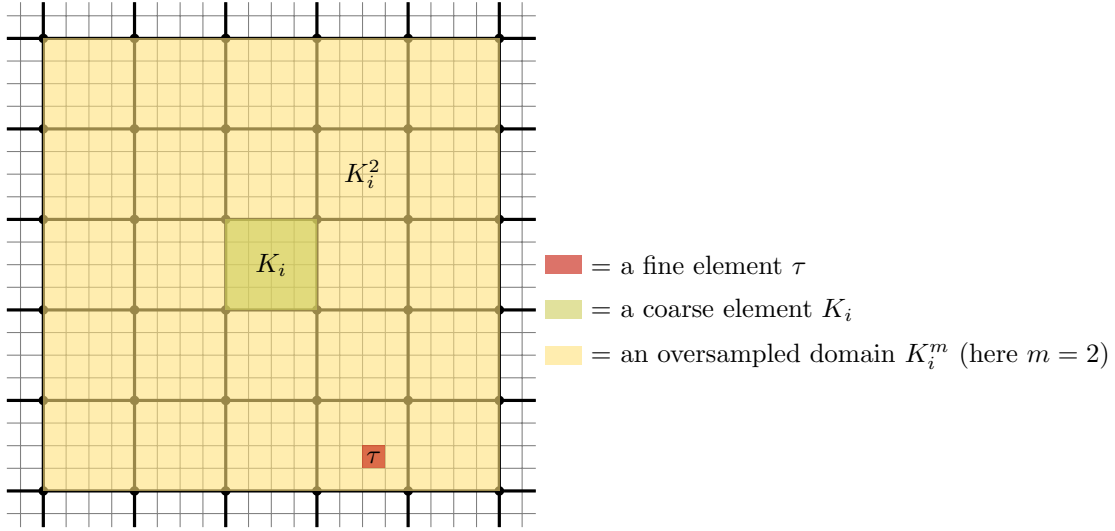
\begin{figure}[!ht]
  \centering
  \begin{tikzpicture}[scale=1.2]
    \draw[step=0.25, gray, thin] (-0.4, -0.4) grid (5.4, 5.4);
    \draw[step=1.0, black, very thick] (-0.4, -0.4) grid (5.4, 5.4);
    % \draw[MyColor4, very thick] (0.0, 0.0) rectangle (4.0, 4.0);
    \foreach \x in {0,...,5}
    \foreach \y in {0,...,5}{
        \fill (1.0 * \x, 1.0 * \y) circle (1.5pt);
      }

    \fill[MyColor4, opacity=0.6] (0.0, 0.0) rectangle (5.0, 5.0);
    \fill[MyColor1, opacity=0.6] (2.0, 2.0) rectangle (3.0, 3.0);
    \fill[MyColor2, opacity=0.6] (3.5, 0.25) rectangle (3.75, 0.5);
    % \fill[MyColor3, opacity=0.6] (-0.25, 0.0) rectangle (0.0, 4.0);
    % \fill[MyColor3, opacity=0.6] (0.0, 0.0) rectangle (4.0, -0.25);
    % \fill[MyColor3, opacity=0.6] (0.0, 4.0) rectangle (4.0, 4.25);
    % \fill[MyColor3, opacity=0.6] (4.25, 0.0) rectangle (4.0, 4.0);

    \node at (2.5, 2.5) {$K_i$};
    \node at (3.5, 3.5) {$K^2_i$};
    \node at (3.625, 0.375) {$\tau$};

    % \fill[MyColor3, opacity=0.6] (4.5, 1.9) rectangle (4.9, 2.1);
    % \node[right] at (4.9, 2.0) {=\ ghost layers};

    \fill[MyColor2, opacity=0.6] (5.5, 2.4) rectangle (5.9, 2.6);
    \node[right] at (5.9, 2.5) {=\ a fine element $\tau$};

    \fill[MyColor1, opacity=0.6] (5.5, 1.9) rectangle (5.9, 2.1);
    \node[right] at (5.9, 2.0) {=\ a coarse element $K_i$};

    \fill[MyColor4, opacity=0.6] (5.5, 1.4) rectangle (5.9, 1.6);
    \node[right] at (5.9, 1.5) {=\ an oversampled domain $K^m_i$ (here $ m=2 $)};
  \end{tikzpicture}
  \caption{Illustration of meshes, fine element, coarse element, and oversampling domain.}
  \label{fig:grid}
\end{figure}
For this quadrilateral mesh, there are 4 vertices contained in an element. We can construct a set of Lagrange bases $  \left\{ {\eta _i^1,\eta _i^2, \eta _i^3 ,\eta _i^4} \right\} $ of the coarse element $K_i \in \mathcal{T}^H$. Then we define $ \tilde \kappa (x) $ piecewise by
\begin{equation}\label{eq:kappa}
  \tilde \kappa (x)=3\sum\limits_{j = 1}^4 {\kappa (x)\nabla \eta _i^j}\cdot \nabla \eta _i^j
\end{equation}
in $ K_i $ which will be used in the following spectral problem.

The process of the construction of CEM-GMsFEM basis functions can be divided into two stages. The first stage involves the construction of the auxiliary space by solving a local spectral problem in each coarse element $ K_i$: find eigen-pairs $ \{ {\lambda _i^j,\phi _i^j} \} $ such that
\begin{equation}\label{eq:eigen_pro}
  {a_i}(\phi _i^j,v) = \lambda _i^j{s_i}(\phi _i^j,v), \  \forall v\in H^1(K_i).
\end{equation}
where
\[ a_i(u, v) =\int_{K_i}\kappa \nabla u \cdot\nabla v \di \bm{x}\ \text{and}\ s_i(u,v)= \int_{{K_i}} \tilde \kappa uv\di\bm{x}. \]
The bilinear form $s(w,v) \coloneqq \int_{{\Omega}} \tilde \kappa wv\di\bm{x} $, and note that $ s(w, v) $ can be well-defined on $ L^2(\Omega ) $. Similarly, we denote the norm
\[ {\norm{v}_s} \coloneqq \sqrt {s(v,v)}\ \text{and}\ {\norm{v}_{s(K_i)}} \coloneqq \sqrt{s_i(v,v)}. \]
Let the eigenvalues $ \{\lambda_i^j\}_{j=1}^\infty $ be arranged in an ascending order, we notice that $\lambda_i^1 = 0 $ always holds. We define the local auxiliary multiscale space $ V_i^{\mathup{aux}} $ by using the first $ l_i $ eigenfunctions
\[V_i^{\mathup{aux}} \coloneqq \Span \CurlyBrackets*{\phi_i^j:1 \leqslant j \leqslant l_i }.\]
It is easy to show that the orthogonal projection $\pi_i$ with respect to the inner product $ s(\cdot , \cdot )$ from $ L^2(K_i) $ to $ V_i^{\mathup{aux}} $ is
\[\pi_i(v)\coloneqq\sum\limits_{j = 1}^{l_i}\frac{s(\phi_i^j,v)}{s(\phi_i^j,\phi_i^j)}\phi_i^j.\]
The global auxiliary space $V^{\mathup{aux}} $ is defined by taking a direct sum of the local auxiliary spaces
$ V^{\mathup{aux}}=\mathop  \oplus_{i = i}^N V_i^{\mathup{aux}}, $
and the global projection is $ \pi \coloneqq \sum^N_{i=1} \pi_i $ accordingly. We can immediately derive the following basic estimates: for all $ v \in H^1(K_i),$
\begin{equation}\label{eq:basic_est_1}
  \norm{v-\pi_i v}^2_{s(K_i)}\leqslant\frac{\norm{v}^2_{a(K_i)}}{\lambda_i^{l_i+1}},
\end{equation}
\begin{equation}\label{eq:basic_est_2}
  \norm{\pi_i v}^2_{s(K_i)}=\norm{v}^2_{s(K_i)}- \norm{v-\pi_i v}^2_{s(K_i)}\leqslant\norm{v}^2_{s(K_i).}
\end{equation}
These estimates show that $ V^{\mathup{aux}} $ can approximate $ V$ satisfyingly and stably with the contrast ratio $ \kappa'/\kappa'' $. However, functions in $ V^{\mathup{aux}} $ may not be continuous in $\Omega$ , and thus $ V^{\mathup{aux}} $ cannot be used as a conforming finite element space. The essential thought in CEM-GMsFEMs is "extending" $\phi^j_i$ into $V$ by solving an energy minimization problem:
\begin{equation}\label{eq:ms_min}
  \psi_i^{j} = \mathup{argmin}\left\{ {a}(\psi , \psi) + s(\pi \psi  - \phi^j_i , \pi \psi  - \phi^j_i ):\ \psi \in V\right\},
\end{equation}
which is a relaxed version of the energy minimization problem \cite{Chung2018}
\[\psi_i^{j}= \mathup{argmin}\left\{ {a}(\psi,\psi): \psi \in V,\  \pi \psi=\phi^j_i \right\}.\]
Moreover, it can be proved that $\psi^j_i$ decays exponentially fast away from $K_i$\cite{Chung2018}\cite{Ye2023}, which implies that solving the energy minimization problems on oversampling domains $ K_i^m $ is reasonable in the next stage.

In the second stage, we denote the space
\[ V_k \coloneqq\left\{v \in H^1(\Omega): \  v=0\text{ on } \TD \cup \mathcal{A}_{k}\right\}\]
to add the restriction on active set and the space
\[ V_k(K_i^m)\coloneqq\left\{v \in H^1(K_i^m): \  v=0\text{ on } \TD\cap {\partial K_i^m}\ \text{or}\ \mathcal{A}_{k}\cap {\partial K_i^m}\ \text{or}\ \Omega\cap{\partial K_i^m} \right\}, \]
for each oversampling domain $ K_i^m $ with the restriction of $V_k$ in $ K_i^m $. Then the multiscale basis functions are defined by
\begin{equation}\label{eq:ms_min_oversampled}
  \psi_{i,k}^{j,m} = \mathup{argmin}\left\{ {a}(\psi , \psi) + s(\pi \psi  - \phi^j_i , \pi \psi  - \phi^j_i ):\ \psi \in V_k(K_i^m) \right\},
\end{equation}
We note that \cref{eq:ms_min_oversampled} is equivalent to: find $ \psi_{i,k}^{j,m} $ such that
\begin{equation}\label{eq:char_ms_min}
  {a}(\psi_{i,k}^{j,m} , v) + s(\pi \psi_{i,k}^{j,m} , \pi v) = s(\phi^j_i , \pi v),\  \forall v \in  V_k(K_i^m).
\end{equation}
Then the CEM-GMsFEM multiscale finite element space is defined by
\[ V_{k,\mathup{ms}}^m=\mathup{span}\left\{\psi_{i,k}^{j,m}:1 \leqslant j \leqslant l_i, 1 \leqslant i \leqslant N \right\}. \]
Similarly, we obtain the global multiscale basis by solving the equivalent problem to \cref{eq:ms_min}: find $ \psi_{i,k}^{j} $ such that
\begin{equation}\label{eq:char_ms_min_glo}
  {a}(\psi_{i,k}^{j} , v) + s(\pi \psi_{i,k}^{j} , \pi v) = s(\phi^j_i , \pi v),\  \forall v \in  V_k.
\end{equation}
The global-version multiscale finite element space is defined by
$$V_{k,\mathup{ms}}^{\mathup{glo}}=\mathup{span}\left\{\psi_{i,k}^{j}:1 \leqslant j \leqslant l_i, 1 \leqslant i \leqslant N \right\}$$ and will be utilized later in the analysis.

The error analysis theory of the original CEM-GMsFEM strongly relies on the existence of $ L^2 $ source term, we might consider introducing function ${\cal N}_i^m{p}$ by imitating the construction of multiscale bases. This approach aims to extend the applicability of the CEM-GMsFEM to cases where an $L^2$ source term is not readily available or suitable. The function ${\cal N}_i^m{p}$ is defined by solving the following local problem (\cref{eq:corrector}) in the oversampling domain $K_i^m$. We can see that the zero-extension of a function in $ V_k(K_i^m) $ still belongs to $ V_k $. It has been proved the solutions corresponding to $\int_{\partial {K_i} \cap \TN } {pv}\di\sigma$ as the RHS may also decay fast away from $K_i$.

Based on the motivations above, at the $k$-th iteration the computational method for obtain the ($k$+1)-th numerical solution consists of four steps:
\begin{description}
  \item[STEP 1] Find ${\cal N}_{i,k+1}^m{p} \in V_k(K_i^m)$ such that,
        \begin{equation}\label{eq:corrector}
          {a}\left( {{\cal N}_{i,k+1}^m{p},v} \right) + s\left( {\pi {\cal N}_{i,k+1}^m{p},\pi v} \right) = \int_{\partial {K_i} \cap \TN } {pv}\di\sigma,\  \forall v \in V_k(K_i^m),
        \end{equation}
        then do the summations as
        ${{\cal N}_{k+1}^m}{p} = \sum\limits_{i = 1}^N {{\cal N}_{i,k+1}^m} {p}.$
  \item[STEP 2] Construct the auxiliary space $V^{\mathup{aux}} $ by \cref{eq:eigen_pro} and the multiscale function space $V_{k,\mathup{ms}}^m$ by \cref{eq:char_ms_min}.
  \item[STEP 3] Solve ${w_{k+1}^m} \in V_{k,\mathup{ms}}^m$ such that for all $v \in V_{k,\mathup{ms}}^m$,
        \begin{equation}\label{eq:weak Neumann}
          {a}\left( {{w_{k+1}^m},v} \right) = \int_\Omega f v\di\bm{x} + \int_{{\Gamma _{\rm{N}}}} {{p}} v\di\sigma - {a}\left( {{{\cal N}_{k+1}^m}{p},v} \right).
        \end{equation}
  \item[STEP 4] Construct the numerical solution to approximate the real solution as
        \[u_{k+1}^m \approx {w_{k+1}^m} + {{\cal N}_{k+1}^m}{p}.\]
\end{description}

%合并修改
%在基函数迭代过程中，不需要更新auxiliary space.
%corrector需要更新
Note that there are $(m+1)N$ oversampling domains whose intersection with the contact boundary $\TC$ is not the empty set. It is only necessary to update the function space $V_k(K_i^m)$ for solving \cref{eq:char_ms_min} on no more than $(m+1)N$ oversampling domains $K_i^m$, i.e., reset the degrees of freedom according to the active set $\mathcal{A}_{k}$. During the iteration process, the auxiliary space remains unchanged. Moreover, it is not necessary to update all the multiscale basis functions; only a subset of the Neumann boundary correctors and multiscale basis functions require updating. Note that \cref{eq:eigen_pro}, \cref{eq:char_ms_min}, \cref{eq:corrector}, and \cref{eq:weak Neumann} are all solved on a fine mesh $\mathcal{T}^h$, which is a refinement of $\mathcal{T}^H$. For brevity, we will not explicitly point this out here and in the following analysis parts. From the computational steps presented above, the multiscale finite element space $V_{k,\mathup{ms}}^m$ is reusable for different source terms, which will greatly accelerate computations in simulations. Moreover, if several particular boundary values (i.e., $p$) that we are interested in admit a low-dimensional structure, it is also possible to build abstract operators ${\cal N}^m$ to achieve an overall saving in computational resources.

\section{Analysis}\label{sec:anal}
In this section, we will analyze the convergence of the iterative CEM-GMsFEM method. We will first show that the multiscale basis functions constructed in the CEM-GMsFEM method can approximate the solution of the contact problem with high contrast coefficients. Then we will prove the convergence of the iterative CEM-GMsFEM method. We will also present the error estimate of the numerical solution obtained by the CEM-GMsFEM method.
\begin{proposition}\label{prop:The stopping rule}
  Let $\mathcal{A}_{k,\mathup{cem}}$ be the fine node index set of active set mapped from the multiscale space on contact boundary $\TC$, and let $\mathcal{A}_{k,h}$ be the fine node index set of active set, with C as the fine node index set on contact boundary $\TC$. Defined $\mathcal{A}_{k,h}$ as follows $\mathcal{A}_{k,h}\coloneqq\left\{x_i\in C: \lambda_{k}^i+cu_{k}^i> 0\right\}$ and let $\mathcal{I}_{k,h}\coloneqq\left\{x_i\in C: \lambda_k^i+cu_k^i\leqslant0\right\}=\left\{x_i\in C\right\} \setminus \mathcal{A}_{k,h}$.
  Assume $\mathcal{A}_{k,h}=\mathcal{A}_{k,\mathup{cem}}$, when $\mathcal{A}_{k+1,h}=\mathcal{A}_{k,h}$, $(u_k, \lambda_k)$ is the solution of the iteration (\cref{eq:iteration-lambda}), the iteration converges in finite steps.
\end{proposition}
\begin{proof}
  Choosing $ x_i \in \mathcal{A}_{k,h}$ we have $u_{k+1} = 0$. Similarly, for $x_i \in \mathcal{I}_{k,h}$ it follows that $\lambda_{k+1} = 0$.
  Since $\mathcal{A}_{k+1,h} = \mathcal{A}_{k,h}$ and $\mathcal{I}_{k+1,h} = \mathcal{I}_{k,h}$, we obtain $\lambda_{k+1}+cu_{k+1}> 0$, which implies that $\lambda_{k+1} >0$ for $x \in \mathcal{A}_{k+1,h}$. Meanwhile $\lambda_{k+1}+cu_{k+1}\leqslant 0$, thus $u_{k+1}\leqslant 0$, for $x \in \mathcal{I}_{k+1,h}$. Together with \cref{eq:iteration-lambda}, it proves that the iterate upon termination of \cref{algo} satisfies $F(\lambda_{k+1})= 0$ for $x_i \in C$.

  Given that $\mathcal{A}_{k,h}=\mathcal{A}_{k,\mathup{cem}}$ for all $k\geqslant0$, the successful termination occurs after a finite number of iterations, since there exists only a finite number of choices for $\mathcal{A}_k$ (and similarly for $\mathcal{I}_k$), which guarantees finite step convergence.
\end{proof}

Assume that the iteration stops at step $k_0$, which means $\mathcal{A}_{k_0-1}=\mathcal{A}_{k_0}$, then we can obtain the corresponding $V_{k_0}(K_i^m),\ V_{k_0,\mathup{ms}}^m$ and $V_{k_0,\mathup{ms}}^{\mathup{glo}}$.

\begin{lemma}\label{lemma:orthogonal}
  (see \cite{Chung2018}) Let $v\in V_{k_0,\mathup{ms}}^{\mathup{glo}}$, then $a(v,v')=0$ for any $v'\in V$ with $\pi v'=0$. Moreover, if there exists $v\in V$ such that $a(v,v')=0$ for any $v'\in V_{k_0,\mathup{ms}}^{\mathup{glo}}$ then $\pi v'=0$.
\end{lemma}
For the boundary corrector, we can also define the corresponding globalized operator:
\begin{equation}\label{eq:corrector_glo}
  {a}\left( {{\cal N}_i^\mathup{glo}{p},v} \right) + s\left( {\pi {\cal N}_i^\mathup{glo}{p},\pi v} \right) = \int_{\partial {K_i} \cap \TN } {pv}\di\sigma,\  \forall v \in V.
\end{equation}
Let ${\cal N}^\mathup{glo} \coloneqq \sum_{i=1}^{N}{\cal N}_i^\mathup{glo}$ and $w^\mathup{glo}$ satisfies for all $v\in V_{k_0,\mathup{ms}}^\mathup{glo}$,
\begin{equation}\label{eq:w_glo}
  {a}\left( {{w^\mathup{glo}},v} \right) = \int_\Omega f v\di\bm{x} + \int_{{\Gamma _{\rm{N}}}} {{p}} v\di\sigma - {a}\left( {{{\cal N}^\mathup{glo}}{p},v} \right).
\end{equation}
Then we can obtain the following error estimate.
\begin{theorem}
  See \cite{Ye2023} Let ${\cal N}_i^\mathup{glo}p, w^\mathup{glo}$ be the solutions of \cref{eq:corrector_glo} and \cref{eq:w_glo} respectively. Let $u$ be the solution of the \cref{eq:variational iteration-1}. Then
  \begin{equation}\label{eq:error estimate}
    \norm{w^\mathup{glo}+{\cal N}^\mathup{glo}p-u}_a\leqslant \frac{1}{\sqrt{\Lambda}}\norm{f}_{s^{-1}},
  \end{equation}
  where $$ \norm{f}_{s^{-1}}\coloneqq \mathop {\sup }\limits_{v\in L^2(\Omega)} \frac{\int_\Omega fv \di x}{\norm{v}_s},$$
  $ \Lambda=\min_i \lambda_i^{l_i+1}$, and ${\cal N}^\mathup{glo}p \coloneqq \sum_{i=1}^{N}{\cal N}_i^\mathup{glo}p$.
\end{theorem}
We can also derive the an estimate show that the boundary corrector has the exponential decay property.
\begin{corollary}\label{corollary} (See \cite{Ye2023})
  Let $ m>1 $, there exists positive constants $c$, $ 0<\theta<1 $ such that
  \begin{equation}\label{eq:decay}
    \norm{({\cal N}^\mathup{glo}-{\cal N}^m)p}_a^2+\norm{\pi({\cal N}^\mathup{glo}-{\cal N}^m)p}_s^2\leqslant cC_{\mathup{tr}}^2\theta^{m-1}(m+1)^d\norm{p}_{L^2(\TN)},
  \end{equation}
  where $C_{\mathup{tr}}$ is the norm of the trace operator $ V\to L^2(\TN) $,
  \[C_{\mathup{tr}}\coloneqq\sup\left\{\frac{\norm{v}_{L^2(\TN)}}{\norm{v}_a}:v\in V,v\neq 0\right\}.\]
\end{corollary}

To analyze function space $V_{k_0,\mathup{ms}}^{m}$ and $V_{k_0,\mathup{ms}}^{\mathup{glo}}$, we need to introduce two operators $ \mathcal{G}^{\mathup{glo}}:L^2(\Omega)\to V$ and $ \mathcal{G}^m:L^2(\Omega)\to V$ defined as follows: for any $\phi \in L^2(\Omega)$ , find $ \mathcal{G}_i^{\mathup{glo}}\phi \in V $ such that
\begin{equation}\label{eq:G_glo}
  a(\mathcal{G}_i^{\mathup{glo}}\phi,v)+s(\pi\mathcal{G}_i^{\mathup{glo}}\phi,\pi v)=s(\pi_i \phi,\pi v),\ \forall v\in V,
\end{equation}
and find $ \mathcal{G}_i^m\phi \in V $ such that
\begin{equation}\label{eq:G_m}
  a(\mathcal{G}_i^m\phi,v)+s(\pi\mathcal{G}_i^m\phi,\pi v)=s(\pi_i \phi,\pi v),\ \forall v\in V_{k_0}(K_i^m),
\end{equation}
where $ \mathcal{G}^{\mathup{glo}}\phi = \sum_{i=1}^{N}\mathcal{G}_i^{\mathup{glo}}\phi $ and $ \mathcal{G}^m\phi = \sum_{i=1}^{N}\mathcal{G}_i^m\phi $. The map $ \pi_i $ is onto $ V_i^{\mathup{aux}} $ and by the definition of $V_{k_0,\mathup{ms}}^m$ and $V_{k_0,\mathup{ms}}^{\mathup{glo}}$, we have the range of $ \mathcal{G}^{\mathup{glo}}$ is $ V_{k_0,\mathup{ms}}^{\mathup{glo}} $ and the range of $ \mathcal{G}^m$ is $ V_{k_0,\mathup{ms}}^{m} $.

The following lemma can be viewed as an inverse estimate of $ V^\mathup{aux} $ to $ V $.
\begin{lemma}\label{lemma:inverse}
  (See \cite{Chung2018}) There exists a positive constant $C_\mathup{inv}$ such that for any $ v\in L^2(\Omega) $, there exists $ {\hat v} \in V $ with $ \pi {\hat v}=\pi v $ and $\norm{{\hat v}}_a\leqslant C_\mathup{inv}\norm{\pi v}_a$.
\end{lemma}
We then consider as an approximation of problem \cref{eq:var inequ} the discrete problem of finding $u^m \in K^m$ such that
\begin{equation}\label{eq:var inequ H}
  a(u^m, v^m-u^m) \geqslant L(v^m-u^m), \  \forall v^m \in K^m,
\end{equation}
where $a(\cdot,\cdot)$ is a continuous and coercivity bilinear form on Hilbert space $V$. We can derive the following error estimate.
\begin{theorem}
  Let $u$ and $u^m$ be the solutions of \cref{eq:var inequ} and \cref{eq:var inequ H} respectively. Let the continuous linear map $A \in L(V, V')$ defined by $a(u, v) = (Au, v)$ $\forall u,v \in V$. Finally, suppose that the solution satisfies the regularity condition $f-Au \in W'$. Then, there exists a constant $C > 0$ such that
  \begin{equation}\label{eq:approxiamtion}
    \norm{u-u^m}_a\leqslant \frac{C}{\alpha}\norm{u-v^m}_a+\left\{\frac{2}{\alpha}\norm{f-Au}_{W'}{\left[ {\norm{u-v^m}_{W}+\norm{u^m-v}_{W}} \right]}\right\}^{1/2},
  \end{equation}
\end{theorem}
\begin{proof}
  By the definitions of $u$ and $u^m$, we have
  \begin{equation}
    \begin{aligned}
      a(u, u-v)       & \leqslant (f,u-v), \  \forall v \in K.         \\
      a(u^m, u^m-v^m) & \leqslant (f,u^m-v^m), \  \forall v^m \in K^m. \\
    \end{aligned}
  \end{equation}
  Adding these inequalities and transposing terms, we obtain
  \[a(u, u) + a(u^m, u^m) \leqslant(f, u - v) + (f, u^m - v^m) + a(u, v) + a(u^m, v^m).\]
  Subtracting $a(u, u^m)+a(u^m, u)$ from both sides and grouping terms, we obtain
  \begin{equation}
    \begin{aligned}
      a(u-u^m, u-u^m) & \leqslant (f, u - v^m) + (f, u^m - v)-a(u,u^m-v)-a(u^m,u-v^m)      \\
                      & = (f, u - v^m) + (f, u^m - v)-a(u,u^m-v)-a(u,u-v^m)+a(u-u^m,u-v^m) \\
                      & = (f-Au, u - v^m) + (f-Au, u^m - v)+a(u-u^m,u-v^m)                 \\
    \end{aligned}
  \end{equation}
  by the definition of $A$ and the fact that $K$ and $K^m \in V$. Since by assumption  $f- Au \in W'$, we have, using the continuity and coercivity of the bilinear form $a(\cdot,\cdot)$,  that
  \[\alpha\norm{u-u^m}_a^2\leqslant \norm{f-Au}_{W'}\norm{u-v^m}_{W}+\norm{f-Au}_{W'}\norm{u^m-v}_{W}+C\norm{u-u^m}_a\norm{u-v^m}_a.\]
  Since
  \[C\norm{u-u^m}_a\norm{u-v^m}_a \leqslant \frac{\alpha}{2}\norm{u-u^m}_a^2+\frac{C^2}{2\alpha}\norm{u-v^m}_a^2\]
  we have
  \[\frac{\alpha}{2}\norm{u-u^m}_a^2\leqslant \norm{f-Au}_{W'}\left[{\norm{u-v^m}_{W}+\norm{u^m-v}_{W}}\right]+\frac{C^2}{2\alpha}\norm{u-v^m}_a^2.\]
  The theorem now foliows from Young's inequality (for positive $a,\ b,\ c$, $a < b + c$ implies that $\sqrt{a}<\sqrt{b}+\sqrt{c}$)

\end{proof}
\begin{remark}
  If $K^m\subset K$, then set $v=u^m$ in \cref{eq:approxiamtion}. Then \cref{eq:approxiamtion} reduces to the following error estimate
  \begin{equation}\label{eq:reduce approxiamtion}
    \norm{u-u^m}_a\leqslant \frac{C}{\alpha}\norm{u-v^m}_a+\left\{\frac{2}{\alpha}\norm{f-Au}_{W'}\norm{u-v^m}_{W}\right\}^{1/2},
  \end{equation}
  for $\forall v^m \in K^m$.
\end{remark}
Since
\begin{equation}
  \norm{u-v^m}_a\leqslant \norm{u-u^\mathup{glo}}_a+\norm{u^\mathup{glo}-v^m}_a, \text{ for } \forall v^m \in K^m,
\end{equation}
we now left to estimate the second term $\norm{u^\mathup{glo}-v^m}_a$. Take $v^m={\cal N}^mp+w_*^m$. Since the range of $\mathcal{G}^\mathup{glo}$ is $V_{k_0,\mathup{ms}}^{\mathup{glo}}$, we are able to find $\phi_* \in L^2(\Omega)$ such that $ w^{\mathup{glo}}=\mathcal{G}^{\mathup{glo}}\phi_*$, then set $w_*^m=\mathcal{G}^m \phi_*$, we derive that
\begin{equation}
  \begin{aligned}
    \norm{u^\mathup{glo}-v^m}_a & \leqslant \norm{(\mathcal{G}^\mathup{glo}-\mathcal{G}^m)\phi_*}_a+\norm{({\cal N}^\mathup{glo}-{\cal N}^m)p}_a. \\
  \end{aligned}
\end{equation}
From \cref{corollary}, we can obtain that
\begin{equation}\label{eq:decay-1}
  \norm{(\mathcal{G}^\mathup{glo}-\mathcal{G}^m)\phi_*}_a^2+\norm{\pi(\mathcal{G}^\mathup{glo}-\mathcal{G}^m)\phi_*}_s^2\leqslant cC_{\mathup{ol}}^2\theta^{m-1}(m+1)^d\norm{\pi \phi_*}_s^2,
\end{equation}
The \cref{eq:G_glo} yields a variational form for $ \mathcal{G}^{\mathup{glo}} $ for all $ v\in V $,
\begin{equation}
  a(w^\mathup{glo},v)+s(\pi w^\mathup{glo},\pi v)=s(\pi \phi_*,\pi v).
\end{equation}
According to \cref{lemma:inverse}, it is easy to find $ \hat{\phi}_*\in V $ such that $ \pi \hat{\phi}_*=\pi \phi_* $ and $ \norm{\hat{\phi}_*}_a\leqslant C_\mathup{inv}\norm{\pi \phi_*}_a $. Replacing $ v $ by $ \hat{\phi}_* $ in the above equation, we have
\begin{equation}
  \norm{\pi \phi_*}_s^2=a(w^\mathup{glo},\hat{\phi}_*)+s(\pi w^\mathup{glo},\pi \phi)\leqslant\norm{\pi\phi_*}_s(C_\mathup{inv}\norm{w^\mathup{glo}}_a+\norm{\pi w^\mathup{glo}}_s),
\end{equation}
Thus,
\begin{equation}\label{eq:estimate}
  \norm{u-v^m}_a\leqslant \frac{1}{\sqrt{\Lambda}}\norm{f}_{s^{-1}}+c\sqrt{C_{\mathup{ol}}}\theta^\frac{m-1}{2}(m+1)^\frac{d}{2}\left\{C_\mathup{tr}\norm{p}_{L^2(\TN)}+C_\mathup{inv}\norm{w^\mathup{glo}}_a+\norm{\pi w^\mathup{glo}}_s\right\}.
\end{equation}

% 2.
% map the multiscale space to finite element space
% \[\norm{u-u_h^{\mathup{cem}}}_a=\norm{u-u^{\mathup{cem}}}_a+\norm{u^{\mathup{cem}}-u_h^{\mathup{cem}}}_a\]

% 3.variational inequality
% for all $v,u\in K$
% \begin{equation}\label{eq:deriKe-variational inequality}
%   \begin{aligned}
%     \int_{\Omega}-\nabla(\kappa \nabla u)(v-u)&=\int_{\Omega}f(v-u)\\
%     \int_{\Omega}\kappa \nabla u\cdot\nabla(v-u)&=\int_{\partial \Omega}{\kappa\nabla u}\cdot\bm{n}(v-u)+\int_{\Omega}f(v-u)\\
%     \int_{\Omega}\kappa \nabla u\cdot\nabla(v-u)&=\int_{\TC}\kappa \nabla u\cdot\bm{n}v+\int_{\TN}p(v-u)+\int_{\Omega}f(v-u)\\
%     \int_{\Omega}\kappa \nabla u\cdot\nabla(v-u)&\geqslant \int_{\TN}p(v-u)+\int_{\Omega}f(v-u)\\
%   \end{aligned}
% \end{equation}
% i.e. for all $u,v\in K$,
% \begin{equation}\label{eq:derivevariational}
%   a(u,v-u)\geqslant L(v-u)
% \end{equation}

\section{Numerical experiments}\label{sec:num}
In this section, we will present some numerical experiments to demonstrate that the multiscale method proposed is effective in a high contrast coefficient setting. We set the domain $ \Omega = (0,1) ^2 $. The medium parameter $ \kappa $ has a $ 400 \times 400 $ resolution and only takes two values. The matrix phase value is $\kappa_\mathup{m}=1 $ and the value in the channels and inclusions is $\kappa_\mathup{I} \gg \kappa_\mathup{m}$, as shown in \cref{fig:medium}.
\begin{figure}[!ht]
  \centering
  \begin{subfigure}[b]{0.45\textwidth}
    \centering
    \includegraphics[width=\textwidth]{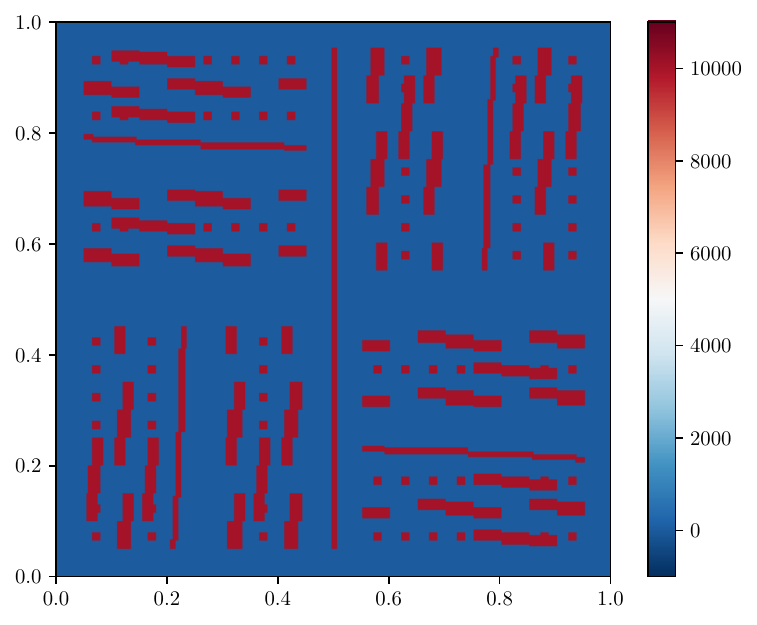}
    \caption{}\label{fig:Medium A}
  \end{subfigure}
  \hfill
  \begin{subfigure}[b]{0.45\textwidth}
    \centering
    \includegraphics[width=\textwidth]{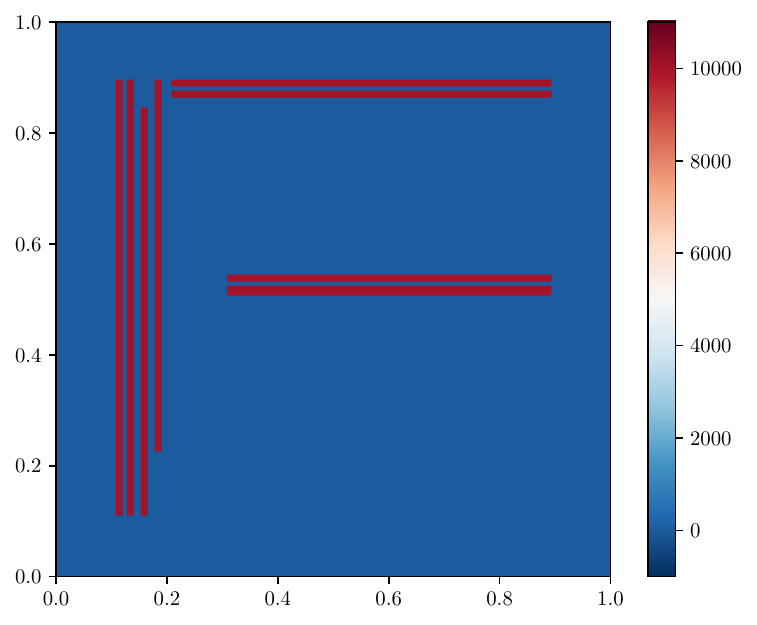}
    \caption{}
    \label{fig:Medium B}
  \end{subfigure}
  % \hfill
  % \begin{subfigure}[b]{0.312\textwidth}
  % 	\centering
  % 	\includegraphics[width=\textwidth]{fig/mediumD.pdf}
  % 	\caption{}
  % 	\label{fig:Medium C}
  % \end{subfigure}
  \caption{The permeability fields (a) medium A; (b) medium B.}\label{fig:medium}
\end{figure}
We define contrast ratios $ \kappa_\mathup{R} \coloneqq\kappa_\mathup{I}/\kappa_\mathup{m}$. For simplifying the implementation, we especially choose $ \tilde \kappa = 24\kappa /H^2 $ instead of the original definition in \cref{eq:kappa}. Moreover, we always fix the number of eigenvectors used to construct auxiliary space $ V^{\mathup{aux}}_i $ as $l_\mathup{m}$, i.e., $ l_1 = l_2 = \cdot \cdot \cdot = l_N = l_\mathup{m} $. The source functions $ f $ are given by $f_1(x,y)=-2x+3y+\sin (2\pi x)\sin (2\pi y)$ and $f_2(x,y)=\frac{1}{2}-x^2+y^2+\cos(\frac{3}{2}\pi x+\pi y)$ for any $ (x,y)\in\Omega $
% ,
% \begin{equation}\label{f2}
% 	f_2=\left\{ \begin{aligned}
% 		10,  & \ 0 < x < 1\text{ and }\frac{3}{8}< y < \frac{5}{8}, \\
% 		10,  & \ \frac{3}{8}< x < \frac{5}{8}\text{ and }0 < y< 1,  \\
% 		-10, & \text{ else},
% 	\end{aligned} \right.
% \end{equation}
% and
% \begin{equation}\label{f3}
% 	f_3=\left\{ \begin{aligned}
% 		10,  & \ 0 < x < 1\text{ and }\frac{1}{2}< y < \frac{3}{4}, \\
% 		-10, & \ \text{ else},
% 	\end{aligned} \right.
% \end{equation}
in \cref{fig:f}.
\begin{figure}[!ht]
  \centering
  \begin{subfigure}[b]{0.45\textwidth}
    \centering
    \includegraphics[width=\textwidth]{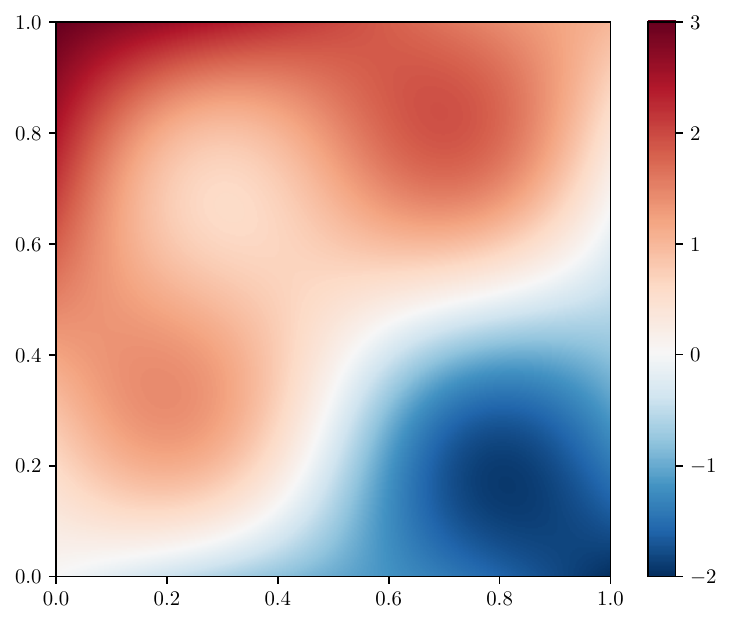}
    \caption{}\label{fig:f_1}
  \end{subfigure}
  \hfill
  \begin{subfigure}[b]{0.45\textwidth}
    \centering
    \includegraphics[width=\textwidth]{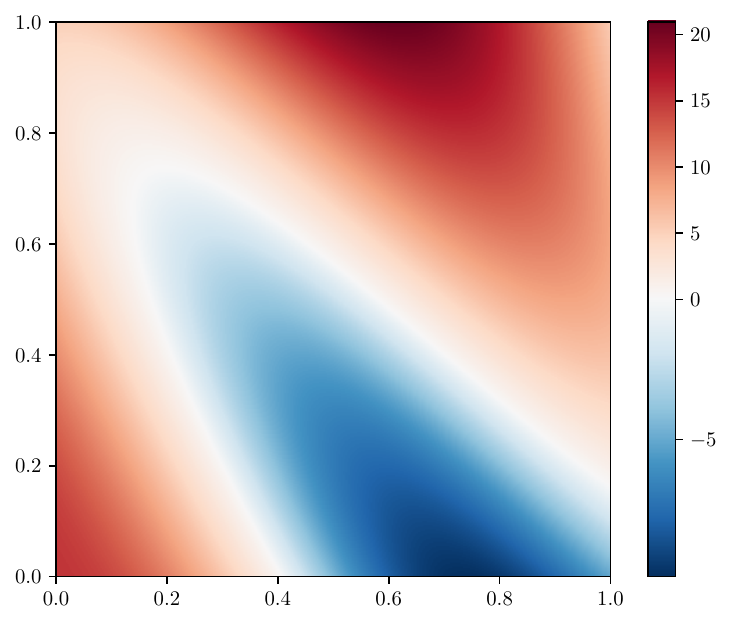}
    \caption{}\label{fig:f_2}
  \end{subfigure}
  % \hfill
  % \begin{subfigure}[b]{0.312\textwidth}
  % 	\centering
  % 	\includegraphics[width=\textwidth]{fig/f_band_imshow.pdf}
  % 	\caption{}\label{fig:f_3}
  % \end{subfigure}
  \caption{The source function (a) $ f_1 $; (b) $ f_2 $.}\label{fig:f}
\end{figure}
The reference solutions $ u^{\mathup{fe}}$ are obtained by the bilinear Lagrange FEM with a fine mesh $ 400 \times 400 $.
% The default parameters are taken as follows: coarse mesh sizes $ H =1/80 $, penalty parameter $ \varepsilon=10^{-4} $, contrast ratios $ \kappa_\mathup{R} = 10^3$, eigenvector numbers $ l_\mathup{m}= 4 $ and oversampling layers $ m=4 $.

Let us obtain the initial value $ u_0 $ by the unconstrained problem on $ \TC$ i.e., $ \TC $ is 0-Neumann boundary value problem, meanwhile $\mathcal{A}_0=\emptyset$ in \cref{algo}, all $ \TC $ is inactive set. After the $k$-th iteration, to evaluate the efficiency of the multiscale method, we consider the relative $ L^2 $ error and energy error between $ u_{k}^{\mathup{cem}}$ and $u_{k}^{\mathup{fe}} $ defined as
\[ E^L_k\coloneqq \frac{\norm{{u^{\mathup{fe}}_{k}} - u^{\mathup{cem}}_{k}}_{{L^2}(\Omega )}}{\norm{{u^{\mathup{fe}}_{k}}}_{{L^2}(\Omega )}}\ \text{and}\ E^a_k\coloneqq \frac{\norm{{u^{\mathup{fe}}_{k}} - u^{\mathup{cem}}_{k}}_a}{\norm{{u^{\mathup{fe}}_{k}}}_a},\]
and the iterative rate for the multiscale solutions  and the reference solutions in two norms defined as
\[ T^{\mathup{cem},L}_k\coloneqq \frac{\norm{{u^{\mathup{cem}}_{k}} - u^{\mathup{cem}}_*}_{{L^2}(\Omega )}}{\norm{{u^{\mathup{cem}}_{k-1}- u^{\mathup{cem}}_*}}_{{L^2}(\Omega )}}\ \text{and}\ T^{\mathup{cem},a}_k\coloneqq \frac{\norm{{u^{\mathup{cem}}_{k}} - u^{\mathup{cem}}_*}_a}{\norm{{u^{\mathup{cem}}_{k-1}- u^{\mathup{cem}}_*}}_a}.\]

\[ T^{\mathup{fe},L}_k\coloneqq \frac{\norm{{u^{\mathup{fe}}_{k}} - u^{\mathup{fe}}_*}_{{L^2}(\Omega )}}{\norm{{u^{\mathup{fe}}_{k-1}- u^{\mathup{fe}}_*}}_{{L^2}(\Omega )}}\ \text{and}\ T^{\mathup{fe},a}_k\coloneqq \frac{\norm{{u^{\mathup{fe}}_{k}} - u^{\mathup{fe}}_*}_a}{\norm{{u^{\mathup{fe}}_{k-1}- u^{\mathup{fe}}_*}}_a}.\]

\subsection{Model problem 1}
We begin with the model problem 1, where the permeability field is medium A and the source function is $ f_1 $. The default parameters are taken as follows: coarse mesh sizes $ H =1/100 $, constant parameter $ c =10 $, contrast ratios $ \kappa_\mathup{R} = 10^3$, eigenvector numbers $ l_\mathup{m}= 4 $ and oversampling layers $ m=4 $.

The exact multiscale solutions $ u^{\mathup{cem}} $ after iterations with medium A are shown in \cref{fig:u_cem}. \cref{fig:u_3D} provides the three-dimensional visualization of the numerical solution, and \cref{fig:u_con} shows the corresponding contour plot. From the figures, it can be clearly observed that the solution maintains a value of zero on the Dirichlet boundary and remains less than or equal to zero on the contact boundary. This demonstrates that our method remains applicable and maintains high accuracy even under complex high-contrast conditions.
\begin{figure}[!ht]
  \centering
  \begin{subfigure}[b]{0.45\textwidth}
    \centering
    \includegraphics[width=\textwidth]{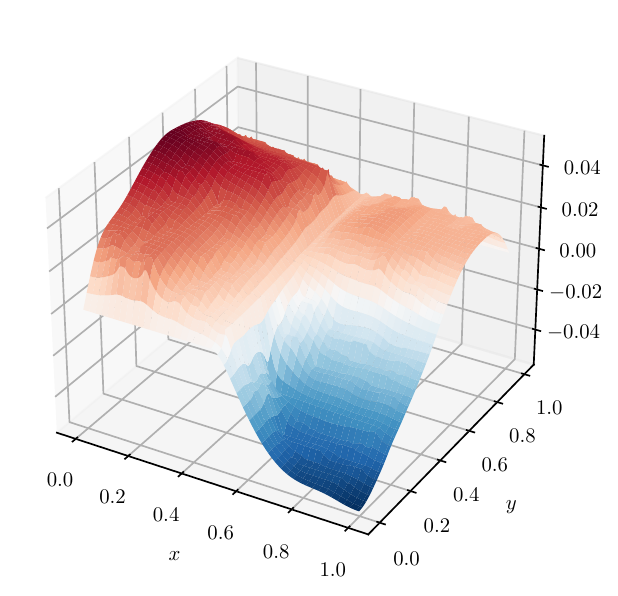}
    \caption{the 3D images}\label{fig:u_3D}
  \end{subfigure}
  \hfill
  \begin{subfigure}[b]{0.45\textwidth}
    \centering
    \includegraphics[width=\textwidth]{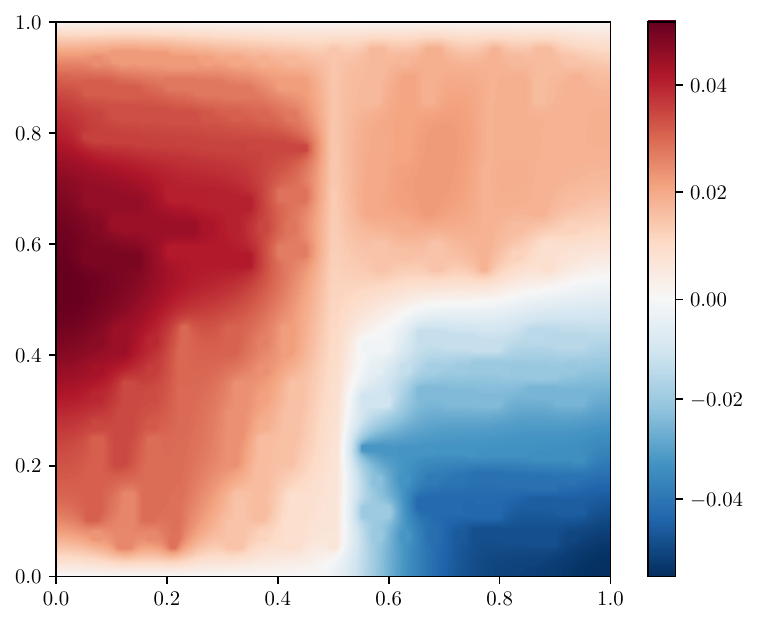}
    \caption{the contour images}\label{fig:u_con}
  \end{subfigure}
  \caption{The multiscale solutions $u^\mathup{cem}$ after iterations with medium A, using the source function $ f_1 $: (a)the 3D images; (b)the contour images; }\label{fig:u_cem}
\end{figure}

The numerical results presented in the \cref{tab:H100m4} demonstrate the convergence behavior and accuracy of the multiscale method. As can be seen from the table, the iterative rate exhibits a monotonic increase in norm values from $k=1$ to 6, followed by a sharp drop to zero at $k=7$, suggesting the iteration terminates after 7 steps when the stopping criterion is met, at which point the active set remains unchanged between consecutive iterations. The iterative rate of multiscale solutions $T_k^{\mathup{cem}}$ and fine-scale solutions $T_k^{\mathup{fe}}$ exhibit nearly identical values for both the $L^2$ norm and energy norm across all iterations, indicating strong agreement between the two methods. The relative $ L^2 $ error and energy error between $ u_{k}^{\mathup{cem}}$ remain remarkably stable, with $ L^2 $ norm errors of order $10^{-6}$ and energy norm errors of order $10^{-3}$, underscoring the method's precision and robustness. These results validate the effectiveness of the proposed method in handling this contact problem while maintaining high accuracy throughout the iterations.

\begin{table}[!ht]
  \caption{Convergence rate and relative error $H=1/100$ and $m=4$}\label{tab:H100m4}
  \centering
  \makegapedcells
  \footnotesize{
    \begin{tabular}{c l c c c c c c c}
      \toprule
      \multicolumn{2}{c}{$ k $ } & \num{1}              & \num{2}           & \num{3}           & \num{4}           & \num{5}           & \num{6}           & \num{7}                               \\

      \toprule
      \multirow{2}{*}{$T_k^\mathup{cem}$}
                                 & $\norm{\cdot}_{L^2}$ & \num{6.54683e-02} & \num{1.51615e-01} & \num{1.76431e-01} & \num{2.53363e-01} & \num{2.81768e-01} & \num{3.16505e-01} & \num{0.00000e+00} \\
                                 & $\norm{\cdot}_{{a}}$ & \num{1.89735e-01} & \num{2.46969e-01} & \num{2.64997e-01} & \num{3.52091e-01} & \num{3.83587e-01} & \num{4.20175e-01} & \num{0.00000e+00} \\

      \midrule
      \multirow{2}{*}{$T_k^\mathup{fe}$}
                                 & $\norm{\cdot}_{L^2}$ & \num{6.54685e-02} & \num{1.51615e-01} & \num{1.76431e-01} & \num{2.53362e-01} & \num{2.81767e-01} & \num{3.16502e-01} & \num{0.00000e+00} \\
                                 & $\norm{\cdot}_{{a}}$ & \num{1.89736e-01} & \num{2.46968e-01} & \num{2.64997e-01} & \num{3.52091e-01} & \num{3.83581e-01} & \num{4.20167e-01} & \num{0.00000e+00} \\
      \midrule
      \multirow{2}{*}{$E_k$}
                                 & $\norm{\cdot}_{L^2}$ & \num{4.10175e-06} & \num{4.06605e-06} & \num{4.06417e-06} & \num{4.06417e-06} & \num{4.06423e-06} & \num{4.06427e-06} & \num{4.06430e-06} \\
                                 & $\norm{\cdot}_{{a}}$ & \num{1.08520e-03} & \num{1.08388e-03} & \num{1.08344e-03} & \num{1.08335e-03} & \num{1.08332e-03} & \num{1.08332e-03} & \num{1.08331e-03} \\
      \bottomrule
    \end{tabular}
  }
\end{table}

In order to better study the behaviour of the numerical solution $ u^{\mathup{cem}} $ and the Lagrange multiplier $ \lambda $ on the contact boundary, we show in \cref{fig:TC} the iterative process of the multiscale solution $ u_k^{\mathup{cem}} $ and $ \lambda $ on the boundary $ \TC $ in medium A. The initial solution $ u_0 $ is obtained from an unconstrained problem on $ \TC $ with values ranging from -0.04 to 0.10 on the contact boundary. After the first iteration step, we can see that $ u_k^{\mathup{cem}} $ remains less than or equal to 0 at all times and begins to satisfy the constraints, and the range of x-values corresponding to less than 0 increases with the number of iterations and gradually converges to the result of $ k = 7 $. In \cref{fig:lambda}, the value of $ \lambda_k $ oscillates between positive and negative values, eventually converging to $ \lambda_7\geqslant0 $. Meanwhile the results of the 7th iteration were able to show that at the contact boundary $ \lambda u=0 $, the KKT constraints on the boundary were met. Once again, the validity and accuracy of the multiscale method is demonstrated.

\begin{figure}[!ht]
  \centering
  \begin{subfigure}[b]{0.49\textwidth}
    \centering
    \includegraphics[width=\textwidth]{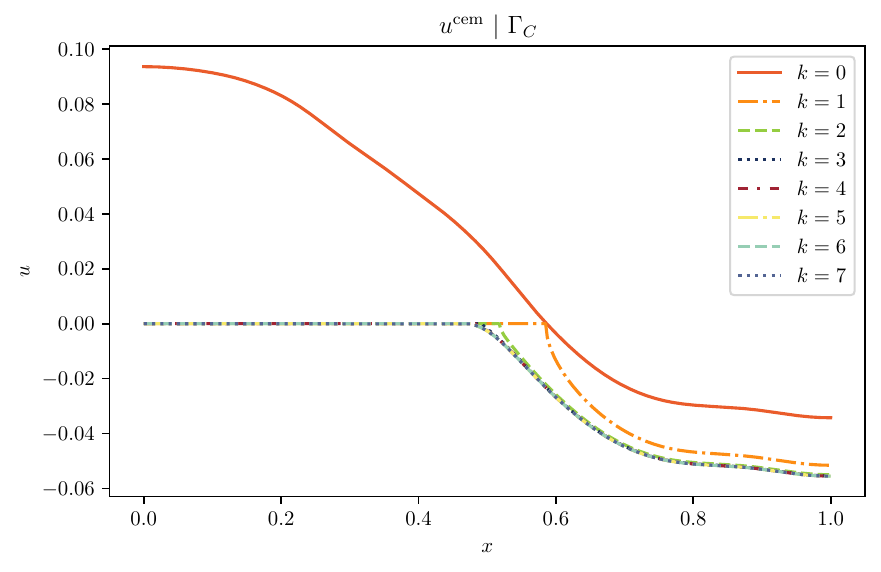}
    \caption{The multiscale solutions $u_k^\mathup{cem}$}\label{fig:u_TC}
  \end{subfigure}
  \hfill
  \begin{subfigure}[b]{0.49\textwidth}
    \centering
    \includegraphics[width=\textwidth]{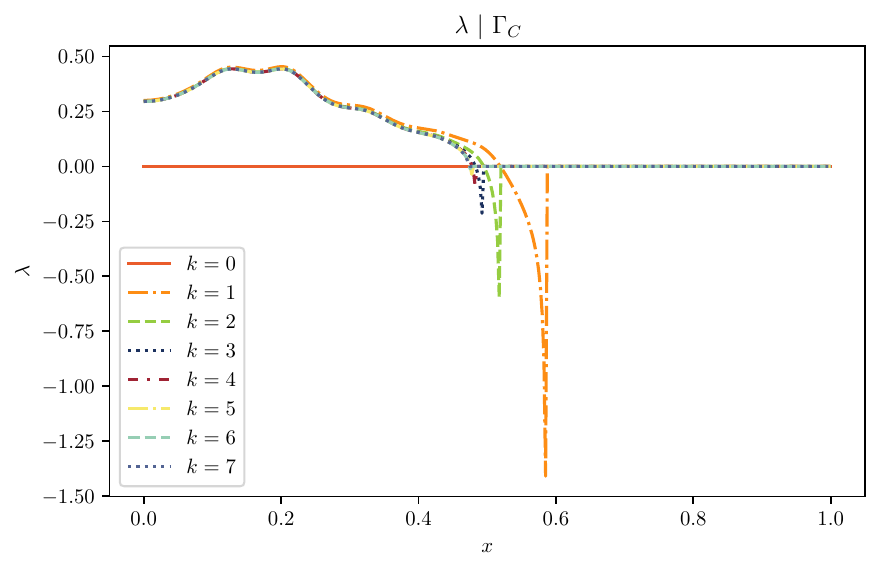}
    \caption{The lagrangian multiplier $\lambda_k$}\label{fig:lambda}
  \end{subfigure}
  \caption{The multiscale solutions $u_k^\mathup{cem}$ and $\lambda_k$ along $\TC$ with medium A and source function $ f_1 $}\label{fig:TC}
\end{figure}

The active sets $ \mathcal{A}_k $ and inactive sets $ \mathcal{I}_k $ are shown in \cref{fig:set}. 1 and 0 represent yes or no active set/inactive set, respectively. From the figure, we can see that the active set gradually decreases during the iteration process, and finally stays unchanged in the 7th iteration, and the iteration is terminated. At the same time, the inactive set gradually increases. Together they form the set $\TC$.

\begin{figure}[!ht]
  \centering
  \begin{subfigure}[b]{0.49\textwidth}
    \centering
    \includegraphics[width=\textwidth]{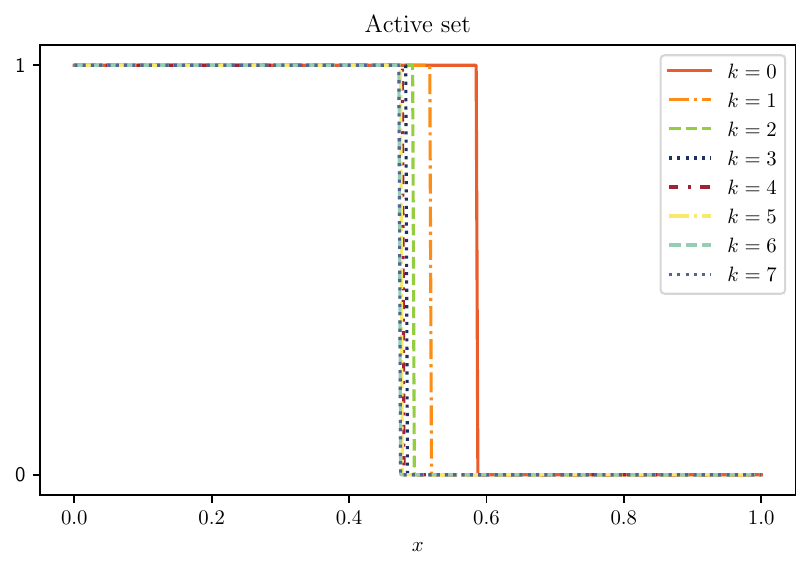}
    \caption{}\label{fig:act}
  \end{subfigure}
  \hfill
  \begin{subfigure}[b]{0.49\textwidth}
    \centering
    \includegraphics[width=\textwidth]{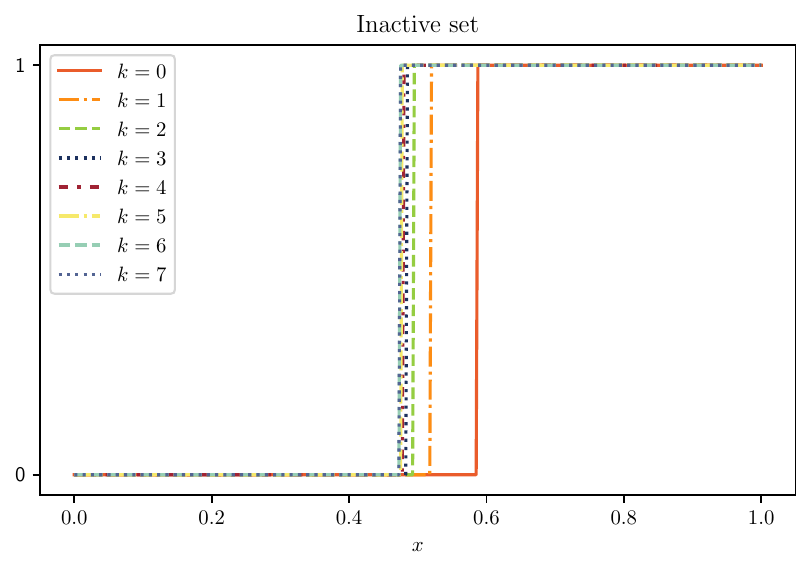}
    \caption{}\label{fig:inact}
  \end{subfigure}
  \caption{The active sets $\mathcal{A}_k$ and inactive sets $\mathcal{I}_k$}\label{fig:set}
\end{figure}

% \begin{figure}
%   \centering
%   \includegraphics[width=\textwidth]{result_fig/err_Af1.pdf}
%   \caption{Convergence rate and error $H=1/100$ and $m=4$}\label{fig:err}
% \end{figure}
\cref{tab:Hm4} presents numerical results for different coarse mesh sizes $H$ across iterations $k = 1$ to 7, comparing relative errors in the $L^2$-norm and energy norm. As coarse mesh get finer, the relative errors in both norms decrease, with the $L^2$-norm error dropping from $3.286\times10^{-4}$ ($H=1/20$) to $4.102\times10^{-6}$ ($H=1/100$) and the energy norm error decreasing  from $9.339\times10^{-3}$($H=1/20$) to $1.085\times10^{-3}$ ($H=1/100$). We can also observe that for all coarse mesh sizes $H$, errors stabilize after the third iteration. These results demonstrate that the multiscale method is effective in solving the model problem 1, with high accuracy across all iterations.

\begin{table}[!ht]
  \caption{The relative error $E_k$ with different coarse mesh size $H$}\label{tab:Hm4}
  \centering
  \makegapedcells
  \footnotesize{
    \begin{tabular}{c l c c c c c c c}
      \toprule
      \multicolumn{2}{c}{$ k $ } & \num{1}              & \num{2}           & \num{3}           & \num{4}           & \num{5}           & \num{6}           & \num{7}                               \\

      \toprule
      % \multirow{2}{*}{$T_k^\mathup{cem}$}
      %                            & $\norm{\cdot}_{L^2}$ & \num{6.54683e-02} & \num{1.51615e-01} & \num{1.76431e-01} & \num{2.53363e-01} & \num{2.81768e-01} & \num{3.16505e-01} & \num{0.00000e+00} \\
      %                            & $\norm{\cdot}_{{a}}$ & \num{1.89735e-01} & \num{2.46969e-01} & \num{2.64997e-01} & \num{3.52091e-01} & \num{3.83587e-01} & \num{4.20175e-01} & \num{0.00000e+00} \\

      % \midrule
      % \multirow{2}{*}{$T_k^\mathup{fe}$}
      %                            & $\norm{\cdot}_{L^2}$ & \num{6.54685e-02} & \num{1.51615e-01} & \num{1.76431e-01} & \num{2.53362e-01} & \num{2.81767e-01} & \num{3.16502e-01} & \num{0.00000e+00} \\
      %                            & $\norm{\cdot}_{{a}}$ & \num{1.89736e-01} & \num{2.46968e-01} & \num{2.64997e-01} & \num{3.52091e-01} & \num{3.83581e-01} & \num{4.20167e-01} & \num{0.00000e+00} \\
      \multirow{2}{*}{$H=1/20$}
                                 & $\norm{\cdot}_{L^2}$ & \num{3.28648e-04} & \num{3.26119e-04} & \num{6.23599e-04} & \num{3.26064e-04} & \num{3.26070e-04} & \num{3.26072e-04} & \num{3.26073e-04} \\

                                 & $\norm{\cdot}_{{a}}$ & \num{9.33881e-03} & \num{9.27993e-03} & \num{1.05797e-02} & \num{9.27624e-03} & \num{9.27621e-03} & \num{9.27621e-03} & \num{9.27621e-03} \\
      \midrule
      \multirow{2}{*}{$H=1/40$}
                                 & $\norm{\cdot}_{L^2}$ & \num{8.21180e-05} & \num{8.14854e-05} & \num{8.14661e-05} & \num{8.14688e-05} & \num{8.14700e-05} & \num{8.14704e-05} & \num{8.14706e-05} \\
                                 & $\norm{\cdot}_{{a}}$ & \num{3.91100e-03} & \num{3.88633e-03} & \num{3.88484e-03} & \num{3.88474e-03} & \num{3.88473e-03} & \num{3.88473e-03} & \num{3.88473e-03} \\
      \midrule
      \multirow{2}{*}{$H=1/80$}
                                 & $\norm{\cdot}_{L^2}$ & \num{1.26706e-05} & \num{1.25743e-05} & \num{1.25713e-05} & \num{1.25717e-05} & \num{1.25719e-05} & \num{1.25720e-05} & \num{1.25720e-05} \\
                                 & $\norm{\cdot}_{{a}}$ & \num{1.29044e-03} & \num{1.28415e-03} & \num{1.28364e-03} & \num{1.28358e-03} & \num{1.28357e-03} & \num{1.28357e-03} & \num{1.28357e-03} \\
      \midrule
      \multirow{2}{*}{$H=1/100$}
                                 & $\norm{\cdot}_{L^2}$ & \num{4.10175e-06} & \num{4.06605e-06} & \num{4.06417e-06} & \num{4.06417e-06} & \num{4.06423e-06} & \num{4.06427e-06} & \num{4.06430e-06} \\
                                 & $\norm{\cdot}_{{a}}$ & \num{1.08520e-03} & \num{1.08388e-03} & \num{1.08344e-03} & \num{1.08335e-03} & \num{1.08332e-03} & \num{1.08332e-03} & \num{1.08331e-03} \\
      \bottomrule
    \end{tabular}
  }
\end{table}

In \cref{tab: H20m}, we present the numerical results for different oversampling layers $m$ with a coarse mesh size of $H=1/20$ and different iteration steps. The results indicate that as the number of oversampling layers increases, the relative errors in both the $L^2$-norm and energy norm decrease. At seventh iteration, $L^2$-norm error drop from $5.545\times 10^{-3}$ to $3.26\times 10^{-4}$ and for energy norm, errors decline from $7.233\times 10^{-3}$ to $9.276\times 10^{-4}$ when oversampling layers $m$ increase from 2 to 4. This demonstrates that increasing the number of oversampling layers enhances the accuracy of the multiscale method.
\begin{table}[!ht]
  \caption{The relative error $E_k$ with different oversampled layers $m$}\label{tab: H20m}
  \centering
  \makegapedcells
  \footnotesize{
    \begin{tabular}{c l c c c c c c c}
      \toprule
      \multicolumn{2}{c}{$ k $ } & \num{1}              & \num{2}           & \num{3}           & \num{4}           & \num{5}           & \num{6}           & \num{7}                               \\

      \toprule
      % \multirow{2}{*}{$T_k^\mathup{cem}$}
      %                            & $\norm{\cdot}_{L^2}$ & \num{6.54669e-02} & \num{1.51614e-01} & \num{2.25180e-01} & \num{1.98497e-01} & \num{2.81738e-01} & \num{3.16446e-01} & \num{0.00000e+00} \\

      %                            & $\norm{\cdot}_{{a}}$ & \num{1.89716e-01} & \num{2.46872e-01} & \num{3.18925e-01} & \num{2.91775e-01} & \num{3.82677e-01} & \num{4.18545e-01} & \num{0.00000e+00} \\

      % \midrule
      % \multirow{2}{*}{$T_k^\mathup{fe}$}
      %                            & $\norm{\cdot}_{L^2}$ & \num{6.54685e-02} & \num{1.51615e-01} & \num{1.76431e-01} & \num{2.53362e-01} & \num{2.81767e-01} & \num{3.16502e-01} & \num{0.00000e+00} \\

      %                            & $\norm{\cdot}_{{a}}$ & \num{1.89736e-01} & \num{2.46968e-01} & \num{2.64997e-01} & \num{3.52091e-01} & \num{3.83581e-01} & \num{4.20167e-01} & \num{0.00000e+00} \\
      \multirow{2}{*}{$m=2$}
                                 & $\norm{\cdot}_{L^2}$ & \num{5.76504e-03} & \num{5.56512e-03} & \num{5.54421e-03} & \num{5.54305e-03} & \num{5.54379e-03} & \num{5.54447e-03} & \num{5.54535e-03} \\

                                 & $\norm{\cdot}_{{a}}$ & \num{7.25422e-02} & \num{7.23876e-02} & \num{7.23426e-02} & \num{7.23333e-02} & \num{7.23308e-02} & \num{7.23301e-02} & \num{7.23298e-02} \\
      \midrule
      \multirow{2}{*}{$m=3$}
                                 & $\norm{\cdot}_{L^2}$ & \num{3.28864e-04} & \num{3.26311e-04} & \num{6.22516e-04} & \num{3.26254e-04} & \num{3.26260e-04} & \num{3.26262e-04} & \num{3.26263e-04} \\

                                 & $\norm{\cdot}_{{a}}$ & \num{9.52300e-03} & \num{9.46179e-03} & \num{1.07391e-02} & \num{9.45743e-03} & \num{9.45737e-03} & \num{9.45736e-03} & \num{9.45736e-03} \\
      \midrule
      \multirow{2}{*}{$m=4$}
                                 & $\norm{\cdot}_{L^2}$ & \num{3.28648e-04} & \num{3.26119e-04} & \num{6.23599e-04} & \num{3.26064e-04} & \num{3.26070e-04} & \num{3.26072e-04} & \num{3.26073e-04} \\

                                 & $\norm{\cdot}_{{a}}$ & \num{9.33881e-03} & \num{9.27993e-03} & \num{1.05797e-02} & \num{9.27624e-03} & \num{9.27621e-03} & \num{9.27621e-03} & \num{9.27621e-03} \\
      \bottomrule
    \end{tabular}
  }
\end{table}

\subsection{Model problem 2}
We then consider model problem 2, where the permeability field is medium B and the source function is $ f_2 $. The default parameters are taken as follows: coarse mesh sizes $ H =1/100 $, constant parameter $ c =10 $, contrast ratios $ \kappa_\mathup{R} = 10^3$, eigenvector numbers $ l_\mathup{m}= 4 $ and oversampling layers $ m=4 $.

\cref{fig:u2_cem} show the contour images and 3D images for the numerical solution $ u^\mathup{cem} $. We can see that the multiscale solution $ u^\mathup{cem} $ is ranged from -0.02 to 0.05, and exhibits characteristics less than or equal to 0 on the contact boundary, which is consistent with the physical context of the problem.

\begin{figure}[!ht]
  \centering
  \begin{subfigure}[b]{0.45\textwidth}
    \centering
    \includegraphics[width=\textwidth]{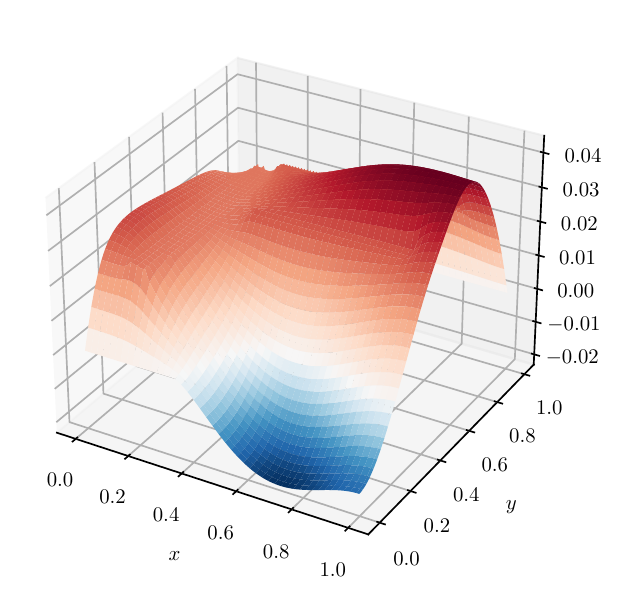}
    \caption{the 3D images}\label{fig:u2_3D}
  \end{subfigure}
  \hfill
  \begin{subfigure}[b]{0.45\textwidth}
    \centering
    \includegraphics[width=\textwidth]{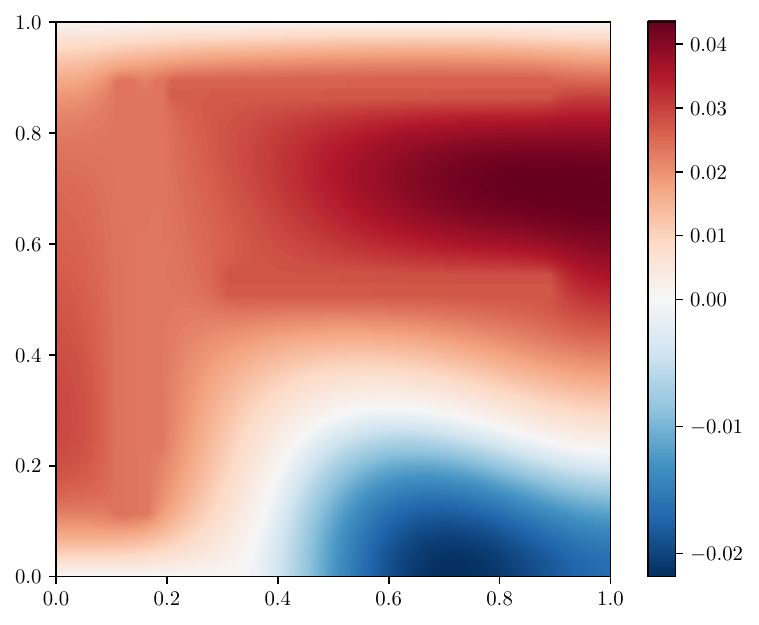}
    \caption{the contour images}\label{fig:u2_con}
  \end{subfigure}
  \caption{The multiscale solutions $u^\mathup{cem}$ after iterations with medium B, using the source function $ f_2 $: (a)the 3D images; (b)the contour images; }\label{fig:u2_cem}
\end{figure}

In \cref{tab:Convergence rate and relative error_model2}, we present the convergence rate and relative error for model problem 2. After seven iterations, the convergence rate converged to zero, indicating that the iteration met the termination condition. These convergence rates did not exhibit a monotonically decreasing trend over the interval $k=1$ to $k=6$, but rather fluctuated. This indicates that the convergence process may be influenced by mesh variations, nonlinear effects, or the multiscale nature of the problem, necessitating further experimental analysis. Under both $L^2$ norm and energy norm, the values of $T_k^\mathrm{cem}$ and $T_k^\mathrm{fe}$ are remarkably close, demonstrating that the CEM-GMsFEM method effectively approximates the finite element solution. The relative error $E_k$ diminishes as the iteration count increases. By the seventh iteration, the $L^2$ norm error had decreased to $1.29676\times 10^{-3}$, whilst the energy norm error had reduced to $2.47525\times 10^{-2}$. This demonstrates that our method effectively captures the solution behaviour near the contact boundary, leading to improved accuracy with each iteration.

\begin{table}[!ht]
  \caption{Convergence rate and relative error for model problem 2}\label{tab:Convergence rate and relative error_model2}
  \centering
  \makegapedcells
  \footnotesize{
    \begin{tabular}{c l c c c c c c c}
      \toprule
      \multicolumn{2}{c}{$ k $ } & \num{1}              & \num{2}           & \num{3}           & \num{4}           & \num{5}           & \num{6}           & \num{7}                               \\

      \toprule
      \multirow{2}{*}{$T_k^\mathup{cem}$}
                                 & $\norm{\cdot}_{L^2}$ & \num{2.13387e-01} & \num{1.81882e-01} & \num{2.16340e-01} & \num{2.35861e-01} & \num{2.08352e-01} & \num{1.53029e-01} & \num{0.00000e+00} \\
                                 & $\norm{\cdot}_{{a}}$ & \num{2.60552e-01} & \num{3.05200e-01} & \num{3.25028e-01} & \num{3.41410e-01} & \num{3.08921e-01} & \num{2.42035e-01} & \num{0.00000e+00} \\

      \midrule
      \multirow{2}{*}{$T_k^\mathup{fe}$}
                                 & $\norm{\cdot}_{L^2}$ & \num{2.13432e-01} & \num{1.81357e-01} & \num{2.15024e-01} & \num{2.33073e-01} & \num{2.03011e-01} & \num{1.43760e-01} & \num{0.00000e+00} \\
                                 & $\norm{\cdot}_{{a}}$  & \num{2.60374e-01} & \num{3.04536e-01} & \num{3.23491e-01} & \num{3.38204e-01} & \num{3.02012e-01} & \num{2.25308e-01} & \num{0.00000e+00} \\
      \midrule
      \multirow{2}{*}{$E_k$}
                                 & $\norm{\cdot}_{L^2}$  & \num{1.29676e-03} & \num{1.41421e-03} & \num{1.46220e-03} & \num{1.48521e-03} & \num{1.49745e-03} & \num{1.50378e-03} & \num{1.50700e-03} \\ 
                                 & $\norm{\cdot}_{{a}}$  & \num{2.62084e-02} & \num{2.50178e-02} & \num{2.48100e-02} & \num{2.47660e-02} & \num{2.47553e-02} & \num{2.47529e-02} & \num{2.47525e-02} \\
      \bottomrule
    \end{tabular}
  }
\end{table}

The multiscale solutions $ u_k^\mathup{cem} $ and Lagrangian multiplier $\lambda_k$ along $\TC$ are shown in \cref{fig:TC2}. The initial value $ u_0^\mathup{cem} $ on the boundary exhibits non-zero components, as it is computed from the unconstrained problem on the contact boundary.
As the number of iterations increases, $ u_k^\mathup{cem} $ on the boundary gradually converges to the result at the termination of the seventh iteration, meaning the value on the contact boundary rapidly converges to a result less than or equal to zero. The solution of $\lambda_k$ exhibits some oscillation during the iteration process, but eventually converges to a result greater than or equal to zero. We can also observe that the solutions of $ u_k^\mathup{cem} $ and $\lambda_k$ satisfy the complementary slackness condition.
This reflects the consistency between theory and experiment, and demonstrates the accuracy of the numerical experiments.

\begin{figure}[!ht]
  \centering
  \begin{subfigure}[b]{0.49\textwidth}
    \centering
    \includegraphics[width=\textwidth]{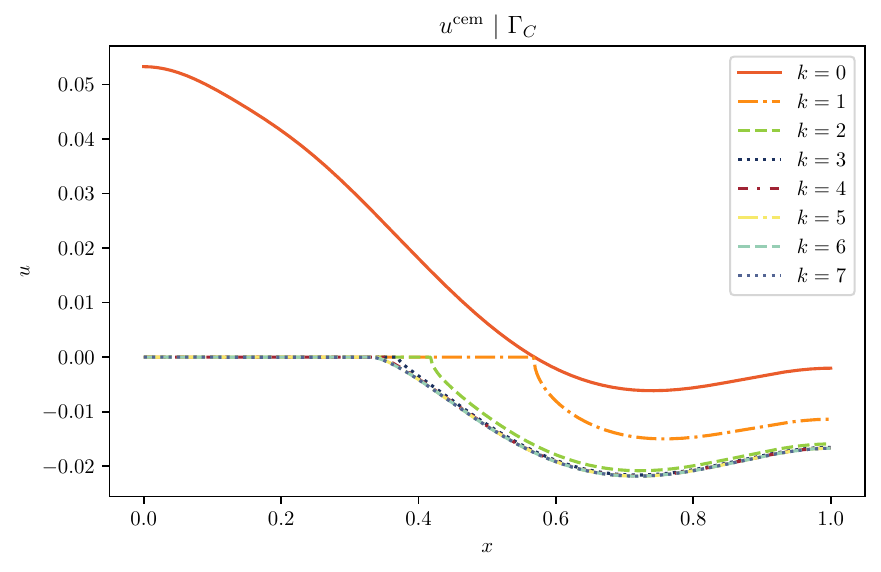}
    \caption{The multiscale solutions $u_k^\mathup{cem}$}\label{fig:u2_TC}
  \end{subfigure}
  \hfill
  \begin{subfigure}[b]{0.49\textwidth}
    \centering
    \includegraphics[width=\textwidth]{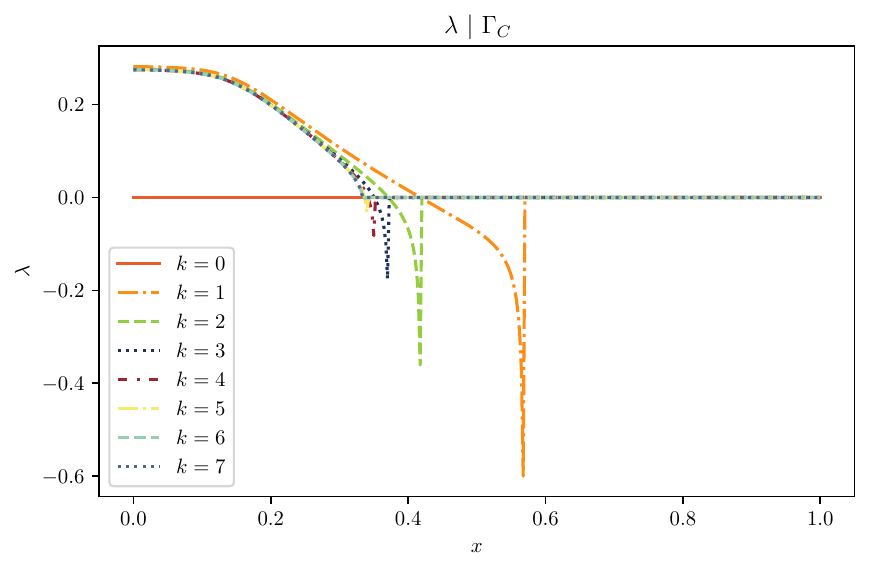}
    \caption{The lagrangian multiplier $\lambda_k$}\label{fig:lambda2}
  \end{subfigure}
  \caption{The multiscale solutions $u_k^\mathup{cem}$ and $\lambda_k$ along $\TC$ with medium B and source function $ f_2 $.}\label{fig:TC2}
\end{figure}

The results for active set $\mathcal{A}_k$ and inactive set $\mathcal{I}_k$ are shown in \cref{fig:set2}. From the figure, we can see that the active set gradually decreases during the iteration process, and finally stays unchanged in the 7th iteration, and the iteration is terminated. At the same time, the inactive set gradually increases. Both exhibit monotonic convergence properties. 

\begin{figure}[!ht]
  \centering
  \begin{subfigure}[b]{0.49\textwidth}
    \centering
    \includegraphics[width=\textwidth]{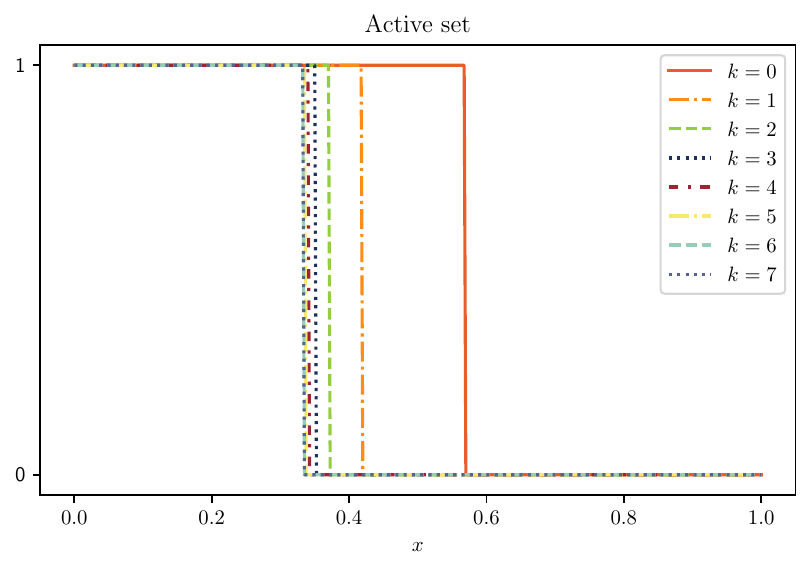}
    \caption{}\label{fig:act2}
  \end{subfigure}
  \hfill
  \begin{subfigure}[b]{0.49\textwidth}
    \centering
    \includegraphics[width=\textwidth]{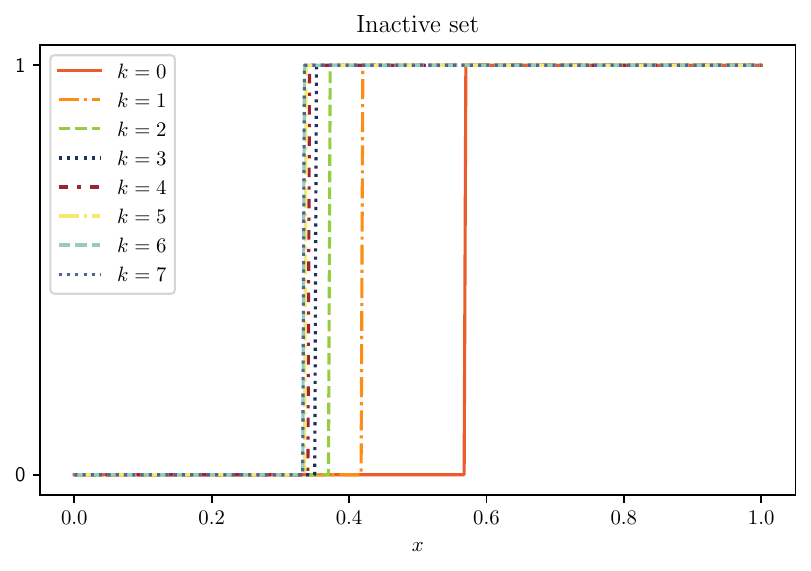}
    \caption{}\label{fig:inact2}
  \end{subfigure}
  \caption{The active sets $\mathcal{A}_k$ and inactive sets $\mathcal{I}_k$}\label{fig:set2}
\end{figure}

\cref{tab:li} The table illustrates the relative error between finite element and numerical solutions, measured under $L^2$ norm and energy norm, with results segmented by the number of eigenvectors used. Overall, the number of eigenvectors is the key factor determining the magnitude of error. When $l_\mathup{m}=1$, the errors under both norm conditions are the largest, reaching approximately 0.78 and 0.69 respectively. As $l_\mathup{m}$ increases to 2, the error decreases significantly by about one order of magnitude. When $l_\mathup{m}$ is further increased to 3 and 4, the error drops to its lowest level, stabilising at $10^{-3}$($L^2$ norm) and $10^{-2}$(energy norm) respectively. Moreover, it is noteworthy that the error values for
$l_\mathup{m}=3$ and $l_\mathup{m}=4$ are remarkably close, indicating that employing three eigenvectors may suffice to meet the precision requirements for this problem.

\begin{table}[!ht]
  \caption{The relative error $E_k$ with different eigenvector numbers $ l_\mathup{m}$}\label{tab:li}
  \centering
  \makegapedcells
  \footnotesize{
    \begin{tabular}{c l c c c c c c c}
      \toprule
      \multicolumn{2}{c}{$ k $ } & \num{1}              & \num{2}           & \num{3}           & \num{4}           & \num{5}           & \num{6}           & \num{7}                               \\

      \toprule

      \multirow{2}{*}{$l_\mathup{m}= 1 $}
                                 & $\norm{\cdot}_{L^2}$ & \num{7.81916e-01} & \num{8.25004e-01} & \num{8.24420e-01} & \num{8.23986e-01} & \num{8.23791e-01} & \num{8.23734e-01} & \num{8.23728e-01} \\ 

                                 & $\norm{\cdot}_{{a}}$ & \num{6.93197e-01} & \num{6.53373e-01} & \num{6.43271e-01} & \num{6.40902e-01} & \num{6.40364e-01} & \num{6.40259e-01} & \num{6.40237e-01} \\ 
      \midrule
      \multirow{2}{*}{$l_\mathup{m}= 2 $}
                                 & $\norm{\cdot}_{L^2}$ & \num{5.55691e-02} & \num{6.34475e-02} & \num{6.40501e-02} & \num{6.45703e-02} & \num{6.45415e-02} & \num{6.45086e-02} & \num{6.44731e-02} \\
                                 & $\norm{\cdot}_{{a}}$ & \num{1.74067e-01} & \num{1.67151e-01} & \num{1.65075e-01} & \num{1.64623e-01} & \num{1.64418e-01} & \num{1.64564e-01} & \num{1.64494e-01} \\ 
      \midrule
      \multirow{2}{*}{$l_\mathup{m}= 3 $}
                                 & $\norm{\cdot}_{L^2}$ & \num{1.29839e-03} & \num{1.41614e-03} & \num{1.46423e-03} & \num{1.48728e-03} & \num{1.49955e-03} & \num{1.50588e-03} & \num{1.50910e-03} \\  
                                 & $\norm{\cdot}_{{a}}$ & \num{2.62269e-02} & \num{2.50357e-02} & \num{2.48276e-02} & \num{2.47837e-02} & \num{2.47730e-02} & \num{2.47706e-02} & \num{2.47702e-02} \\  
      \midrule
      \multirow{2}{*}{$l_\mathup{m}= 4 $}
                                 & $\norm{\cdot}_{L^2}$  & \num{1.29676e-03} & \num{1.41421e-03} & \num{1.46220e-03} & \num{1.48521e-03} & \num{1.49745e-03} & \num{1.50378e-03} & \num{1.50700e-03} \\ 
                                 & $\norm{\cdot}_{{a}}$  & \num{2.62084e-02} & \num{2.50178e-02} & \num{2.48100e-02} & \num{2.47660e-02} & \num{2.47553e-02} & \num{2.47529e-02} & \num{2.47525e-02} \\
      \bottomrule
    \end{tabular}
  }
\end{table}

\cref{tab:kapppa} presents and compares the relative error $E_k$ with respect to the iteration count $k$. Overall, increasing the contrast ratio $\kappa_\mathup{R}$ leads to a significant rise in computational error for both norm metrics. When $\kappa_\mathup{R}$ increases from $10^1$ to $10^4$, the $L^2$ norm error increases from approximately $10^{-5}$ to $10^{-2}$ order of magnitude, while the energy norm error also increased from approximately $10^{-3}$ to $10^{-2}$ order of magnitude, indicating that high contrast poses a significant challenge to computational accuracy. Although the relative error in the energy norm diminishes with increasing iterations, the overall error level remains substantial, particularly under high-contrast conditions. This indicates that finer meshes and additional basis functions may be required to enhance computational accuracy when handling high-contrast permeability fields. Whilst the relative error exhibits greater magnitude under high-contrast condition, it consistently remains within a low order of magnitude ($10^{-2}$). This demonstrates the efficiency and robustness of our approach in handling high-contrast problems.

\begin{table}[!ht]
  \caption{The relative error $E_k$ with different contrast ratio $\kappa_\mathup{R}$}\label{tab:kapppa}
  \centering
  \makegapedcells
  \footnotesize{
    \begin{tabular}{c l c c c c c c c}
      \toprule
      \multicolumn{2}{c}{$ k $ } & \num{1}              & \num{2}           & \num{3}           & \num{4}           & \num{5}           & \num{6}           & \num{7}                               \\

      \toprule

      \multirow{2}{*}{$\kappa_\mathup{R} = 10^1$}
                                 & $\norm{\cdot}_{L^2}$ & \num{5.43394e-05} & \num{5.72611e-05} & \num{6.32962e-05} & \num{6.40563e-05} & \num{6.44618e-05} & \num{6.47259e-05} & \num{6.48176e-05} \\ 

                                 & $\norm{\cdot}_{{a}}$  & \num{5.06876e-03} & \num{4.82272e-03} & \num{4.65099e-03} & \num{4.63913e-03} & \num{4.63646e-03} & \num{4.63575e-03} & \num{4.63566e-03} \\ 
      \midrule
      \multirow{2}{*}{$\kappa_\mathup{R} = 10^2$}
                                 & $\norm{\cdot}_{L^2}$ & \num{1.46723e-04} & \num{1.72031e-04} & \num{1.77728e-04} & \num{1.80903e-04} & \num{1.82306e-04} & \num{1.83038e-04} & \num{1.83412e-04} \\ 
                                 & $\norm{\cdot}_{{a}}$ & \num{9.27815e-03} & \num{8.60485e-03} & \num{8.51492e-03} & \num{8.49431e-03} & \num{8.49030e-03} & \num{8.48926e-03} & \num{8.48900e-03} \\ 
      \midrule
      \multirow{2}{*}{$\kappa_\mathup{R} = 10^3$}
                                 & $\norm{\cdot}_{L^2}$  & \num{1.29676e-03} & \num{1.41421e-03} & \num{1.46220e-03} & \num{1.48521e-03} & \num{1.49745e-03} & \num{1.50378e-03} & \num{1.50700e-03} \\ 
                                 & $\norm{\cdot}_{{a}}$  & \num{2.62084e-02} & \num{2.50178e-02} & \num{2.48100e-02} & \num{2.47660e-02} & \num{2.47553e-02} & \num{2.47529e-02} & \num{2.47525e-02} \\
      \midrule
      \multirow{2}{*}{$\kappa_\mathup{R} = 10^4$}
                                 & $\norm{\cdot}_{L^2}$ & \num{1.25675e-02} & \num{1.37042e-02} & \num{1.46558e-02} & \num{1.46234e-02} & \num{1.46035e-02} & \num{1.45929e-02} & \num{1.45929e-02} \\  
                                 & $\norm{\cdot}_{{a}}$ & \num{8.16963e-02} & \num{7.80578e-02} & \num{7.76199e-02} & \num{7.73486e-02} & \num{7.72830e-02} & \num{7.72699e-02} & \num{7.72699e-02}\\ 
      \bottomrule
    \end{tabular}
  }
\end{table}

\section{Conclusion}
We have developed an iterative multiscale method that seamlessly combines CEM-GMsFEM with a primal-dual active set strategy to tackle high-contrast contact problems.
The proposed scheme builds an auxiliary space by solving local spectral problems and constructs multiscale basis functions through localized energy minimization on oversampled domains.
This approach yields a reduced order model that is robust with respect to high heterogeneities and effectively enforces the non-penetration condition along the contact boundary.
Our theoretical analysis confirms finite step convergence and provides error estimates that account for the contrast ratio and the coarse mesh size.
Numerical experiments validate the method's accuracy and demonstrate its ability to capture the essential features of the solution, with both the multiscale and fine-scale solutions exhibiting consistent convergence behavior.
Future work will focus on extending the framework to dynamic contact problems and exploring parallel implementations to further enhance computational efficiency.

\section*{Acknowledgment}

The research of Eric Chung is partially supported by the Hong Kong RGC General Research Fund (Project: 14305222).

\bibliographystyle{plain}
\bibliography{cemgmsfem.bib}
\end{document}